\documentclass[reqno]{amsart}

\pdfoutput=1

\usepackage{bbm}

\newcommand{\cJ}{\mathcal{J}}

\newcommand{\cO}{\mathcal{O}}

\newcommand{\cS}{\mathcal{S}}
\newcommand{\cU}{\mathcal{U}}

\newcommand{\bF}{\mathbb{F}}

\newcommand{\bQ}{\mathbb{Q}}\newcommand{\bR}{\mathbb{R}}

\newcommand{\bZ}{\mathbb{Z}}

\usepackage{lmodern}
\usepackage[dvipsnames,svgnames,x11names,hyperref]{xcolor}
\usepackage{mathtools,url,graphicx,verbatim,amssymb,stmaryrd,nicefrac}
\usepackage[pagebackref,colorlinks,citecolor=Mahogany,linkcolor=Mahogany,urlcolor=Mahogany,filecolor=Mahogany]{hyperref}
\usepackage{microtype}
\usepackage{enumitem}
\usepackage[margin=1.29in]{geometry}
\usepackage{tikz, tikz-cd}
\usetikzlibrary{matrix,calc,arrows,patterns,decorations.pathreplacing, decorations.markings}
\usepackage[capitalize,noabbrev]{cleveref}
\usepackage{caption}
\usepackage{subcaption}
\newcommand*\circled[1]{\tikz[baseline=(char.base)]{
		\node[shape=circle,draw,inner sep=0.5pt] (char) {#1};}}
	
\AtBeginDocument{%
   \def\MR#1{}}

\newtheorem{theorem}{Theorem}[section]
\newtheorem{lemma}[theorem]{Lemma}
\newtheorem{proposition}[theorem]{Proposition} 
\newtheorem{corollary}[theorem]{Corollary}

\newtheorem*{corollary*}{Corollary}

\newtheorem{atheorem}{Theorem}

\newtheorem{acorollary}[atheorem]{Corollary}

\theoremstyle{definition}

\newtheorem{convention}[theorem]{Convention}
\newtheorem*{convention*}{Convention}
\theoremstyle{remark}
\newtheorem{example}[theorem]{Example} 
 
\newtheorem{remark}[theorem]{Remark}
\newtheorem*{remark*}{Remark}

\newtheorem{question}[theorem]{Question}

\newcommand{\cat}[1]{\mathsf{#1}}
\newcommand{\mr}[1]{{\rm #1}}
\newcommand{\ul}[1]{{\underline{#1}}}

\newcommand{\ra}{\rightarrow}
\newcommand{\lra}{\longrightarrow}
\newcommand{\xlra}[1]{\overset{#1}{\longrightarrow}}

\newcommand{\loc}{\mr{loc}}

\newcommand{\FM}{\mr{FM}}

\newcommand{\Diff}{\mr{Diff}}

\newcommand{\hAut}{\mr{hAut}}
\newcommand{\Map}{\mr{Map}}

\newcommand{\Emb}{\mr{Emb}}
\newcommand{\Bun}{\mr{Bun}}
\newcommand{\Man}{\cat{Man}}

\makeatletter
\newcommand*\bigcdot{\mathpalette\bigcdot@{.7}}
\newcommand*\bigcdot@[2]{\mathbin{\vcenter{\hbox{\scalebox{#2}{$\m@th#1\bullet$}}}}}
\makeatother

\tikzset{
  symbol/.style={
    draw=none,
    every to/.append style={
      edge node={node [auto=false]{$#1$}}}
  }
}

\DeclareMathOperator{\holim}{holim}

\DeclareMathOperator{\hofib}{hofib}

\title{Embedding calculus for surfaces}

\author{Manuel Krannich}
\address{Department of Mathematics, Karlsruhe Institute of Technology, 76131 Karlsruhe, Germany}
\email{krannich@kit.edu}

\author{Alexander Kupers}
\address{Department of Computer and Mathematical Sciences, University of Toronto Scarborough, 1265 Military Trail, Toronto, ON M1C 1A4, Canada}
\email{a.kupers@utoronto.ca}

\begin{document}

\begin{abstract}
We prove convergence of Goodwillie--Weiss' embedding calculus for spaces of embeddings into a manifold of dimension at most two, so in particular for diffeomorphisms between surfaces. We also relate the Johnson filtration of the mapping class group of a surface to a certain filtration arising from embedding calculus.
\end{abstract}

\maketitle 

\tableofcontents

\section{Introduction}
For smooth manifolds $M$ and $N$ and an embedding $e_\partial \colon \partial M \hookrightarrow \partial N$, we write $\Emb_\partial(M,N)$ for the space of embeddings that agree with $e_\partial$ on $\partial M$, equipped with the smooth topology. Embedding calculus \`a la Goodwillie and Weiss provides a space $T_\infty \Emb_\partial(M,N)$ and a map
\begin{equation}\label{equ:emb-calc-approximation}
\Emb_\partial(M,N) \lra T_\infty \Emb_\partial(M,N),
\end{equation}
which approximates the space of embeddings through restrictions to subsets diffeomorphic to a finite collection of open discs and a collar. The space $T_\infty \Emb_\partial(M,N)$ arises as a homotopy limit of a tower of maps whose homotopy fibres have an explicit description in terms of configuration spaces of $M$ and $N$ \cite{WeissImmersion,WeissImmersionErratum}, so its homotopy type is sometimes easier to study than that of $\Emb_\partial(M,N)$. The main result in this context is due to Goodwillie, Klein, and Weiss \cite{GoodwillieWeiss,GoodwillieKlein} and says that if the difference of the dimension of $N$ and the relative handle dimension of the boundary inclusion $\partial M\subset M$ is at least three, then embedding calculus \emph{converges} in the sense that \eqref{equ:emb-calc-approximation} is a weak homotopy equivalence. If this assumption is not met, little is known about for which choices of $M$ and $N$ embedding calculus converges (but see \cref{rem:intro-remark} \ref{rem:disc-structure-space} and \ref{rem:convergence} below).

\subsection{Convergence in low dimensions}
In the first part of this work, we study \eqref{equ:emb-calc-approximation} when the target $N$ has dimension at most two. Our main result shows that embedding calculus always converges under this assumption, even though the assumption on the handle codimension is not satisfied.

\begin{atheorem}\label{athm:main-general} For compact manifolds $M$ and $N$ with $\dim(N)\le 2$, the map
	\[\Emb_\partial(M,N) \lra \ T_\infty \Emb_\partial(M,N)\]
	is a weak homotopy equivalence for any embedding $e_\partial\colon \partial M\hookrightarrow \partial N$.
\end{atheorem}

Perhaps the most interesting (hence eponymous) instance of \cref{athm:main-general} is when $M=N$ is a surface $\Sigma$ and $e_\partial=\mathrm{id}_{\partial \Sigma}$. In this case \cref{athm:main-general} specialises to the following:

\begin{acorollary}\label{acor:main-diff}For a compact surface $\Sigma$, possibly with boundary and non-orientable, the map
	\[\Diff_\partial(\Sigma) \lra \ T_\infty \Emb_\partial(\Sigma,\Sigma)\]
	is a weak homotopy equivalence.
\end{acorollary}

\begin{remark}\label{rem:intro-remark}\ 
\begin{enumerate}[label=(\roman*),leftmargin=*]
\item We prove \cref{athm:main-general} as a special case of a more general result that also treats embeddings spaces of \emph{triads} (see \cref{thm:main-general-triads}).
\item \label{rem:disc-structure-space}\cref{athm:main-general} is special to dimension at most $2$: in \cite{KrannichKupersConvergence}, we show that this results fails for $N=D^3$ and for most high-dimensional compact manifolds $N$. In the language of that paper, \cref{athm:main-general} implies that the smooth $\cat{Disc}$-structure space $\cS_\partial^{\cat{Disc}}(N)$ is contractible if $\dim(N)\le 2$.
\item The proof of \cref{athm:main-general} does not rely on Goodwillie, Klein, and Weiss' convergence results.
\item \cref{athm:main-general} is stronger than \cref{acor:main-diff}, even if $\dim(M)=\dim(N)=2$. It implies that $T_\infty\Emb_\partial(\Sigma,\Sigma')=\varnothing$ if $\Sigma$ and $\Sigma'$ are connected compact surfaces that are not diffeomorphic.
\item Composition induces an $E_1$-structure on $T_\infty \Emb_\partial(M,M)$ with respect to which the map $\Emb_\partial(M,M)\ra T_\infty \Emb_\partial(M,M)$ is an $E_1$-map. For a compact manifold $M$, the $E_1$-space $\Emb_\partial(M,M)=\Diff_\partial(M)$ is grouplike, but it is not known whether the same holds for $T_\infty \Emb_\partial(M,M)$. \cref{athm:main-general} implies that this is the case if $\dim(M)\le 2$.
\item \label{rem:convergence}\cref{athm:main-general} provide a class of examples for which the map $\Emb_\partial(M,N)\ra T_\infty\Emb(M,N)$ is a weak equivalence in handle codimension less than three. A few examples of this form were known before; see \cite[Theorem C, Section 6.2.4]{KnudsenKupers}. In contrast, there are some cases for which it is known that embedding calculus does not converge, such as for $M=D^1$ and $N=D^3$ by an argument due to Goodwillie.
\end{enumerate}
\end{remark}

\subsection{Embedding calculus and the Johnson filtration}\label{sec:intro-johnson}The \emph{Johnson filtration} \[\pi_0\,\Diff_\partial(\Sigma)=\mathcal{J}(0)\supset\mathcal{J}(1)\supset\mathcal{J}(2)\supset\cdots \] of the mapping class group $\pi_0\,\Diff_\partial(\Sigma)$ of an orientable surface $\Sigma$ of genus $g$ with one boundary component is the filtration by the kernels of the action of  $\pi_0\,\Diff_\partial(\Sigma)$ on the quotients of the fundamental group $\pi_1(\Sigma,*)$ based at the point in the boundary, by the constituents of its lower central series. By work of Moriyama \cite{Moriyama}, this filtration can be recovered from the action of $\pi_0\,\Diff_\partial(\Sigma)$ on the compactly supported cohomology of the configuration spaces of the punctured surface $\Sigma\backslash \{*\}$. It is reasonable to expect a relationship between the Johnson filtration and embedding calculus, as the latter may be viewed as the study of embeddings via their induced maps between the homotopy types of configuration spaces of thickened points in source and target.

The second part of this work serves to establish one such a relationship: we introduce a filtration 
\begin{equation}\label{equ:filtration-embcalc}\pi_0\,\Diff_\partial(\Sigma)=T\mathcal{J}_{\nicefrac{\partial}{2}}^{H\mathbb{Z}}(0)\supset T\mathcal{J}_{\nicefrac{\partial}{2}}^{H\mathbb{Z}}(1)\supset T\mathcal{J}_{\nicefrac{\partial}{2}}^{H\mathbb{Z}}(2)\supset\cdots \end{equation}
arising from the cardinality filtration of embedding calculus in $H\mathbb{Z}$-modules applied to the space of self-embeddings fixed on an interval in the boundary (see \cref{sec:johnson} for precise definitions), and we use \cite{Moriyama} to show that this filtration contained in the Johnson filtration:
\[T\mathcal{J}_{\nicefrac{\partial}{2}}^{H\mathbb{Z}}(k)\subset\mathcal{J}(k)\quad\text{for }k\ge0.\]

\subsection*{Acknowledgements}
MK was partially supported by the European Research Council (ERC) under the European Union’s Horizon 2020 research and innovation programme (grant agreement No. 756444), and partially funded by the Deutsche Forschungsgemeinschaft (DFG, German Research Foundation) under Germany's Excellence Strategy EXC 2044 –390685587, Mathematics Münster: Dynamics–Geometry–Structure.AK acknowledges the support of the Natural Sciences and Engineering Research Council of Canada (NSERC) [funding reference number 512156 and 512250], as well as the Research Competitiveness Fund of the University of Toronto at Scarborough. AK is supported by an Alfred J.~Sloan Research Fellowship.

\section{Generalities on spaces of embeddings and embedding calculus}\label{sec:emb-calc}
We begin by fixing some conventions on spaces of embeddings, followed by recalling various known properties of embedding calculus and complementing them with some new properties such as a lemma for lifting embeddings along covering spaces in the context of embedding calculus.

\subsection{Spaces of embeddings and maps}\label{sect:triads} All our manifolds will be smooth and may be noncompact, disconnected, or nonorientable. A \emph{manifold triad} is a manifold $M$ together with a decomposition of its boundary $\partial M=\partial_0M\cup\partial_1M$ into two codimension zero submanifolds that intersect at a set $\partial(\partial_0M)=\partial(\partial_1M)$ of corners. Any of these sets may be empty or disconnected. If this decomposition is not specified, we implicitly take $\partial_0M=\partial M$ and $\partial_1M=\varnothing$. 

When studying embeddings between manifolds triads $M$ and $N$, we always fix a \emph{boundary condition}, i.e.~an embedding $e_{\partial_0} \colon \partial_0 M \hookrightarrow \partial_0 N$, and only consider embeddings $e\colon M\hookrightarrow N$ that restrict to $e_{\partial_0}$ on $\partial_0M$ and have near $\partial_0M$ the form $e_{\partial_0}\times \mr{id}_\mr{[0,1)}\colon \partial_0M\times [0,1)\hookrightarrow \partial_0N\times [0,1)$ with respect to collars of $\partial_0M$ and $\partial_0N$. We denote the space of such embeddings in the weak $\mathcal{C}^\infty$-topology by $\Emb_{\partial_0}(M,N)$. We replace the subscript $\partial_0$ by $\partial$ to indicate if $\partial_0M=\partial M$, and drop the subscript if we want to emphasise if $\partial_0M=\varnothing$ holds. As a final piece of notation, given manifold triads $M$ and $L$, we consider $M\sqcup L$ as manifold triad via $\partial_0(M\sqcup L)=\partial_0M\sqcup \partial_0L$. 

Similarly, we also consider the space of bundle maps $\Bun_{\partial_0}(TM,TN)$. By this we mean the space of fibrewise injective linear maps $TM\ra TN$ that restrict to the derivative $d(e_{\partial_0})$ on $T{\partial_0M}$, in the compact-open topology. Taking derivatives induces a map $\Emb_{\partial_0}(M,N)\ra \Bun_{\partial_0}(TM,TN)$ which we may postcompose with forgetful map $ \Bun_{\partial_0}(TM,TN) \to \Map_{\partial_0}(M,N)$ to the space of continuous maps extending $e_{\partial_0}$, equipped with the compact-open topology.

\subsection{Manifold calculus}\label{sec:manifold-calc}
Given manifold triads $M$ and $N$ and a boundary condition $e_{\partial_0}\colon \partial_0M\hookrightarrow \partial_0N$ as above, Goodwillie--Weiss' \emph{embedding calculus} \cite{WeissImmersion, GoodwillieWeiss} gives a space $T_\infty \Emb_{\partial_0}(M,N)$ (or rather, a homotopy type) together with a map
\begin{equation}\label{eqn:emb-calc-map} 
		\Emb_{\partial_0}(M,N) \lra T_\infty \Emb_{\partial_0}(M,N).
\end{equation}
Embedding calculus \emph{converges} if the map \eqref{eqn:emb-calc-map} is a weak homotopy equivalence (shortened to \emph{weak equivalence} throughout this work). This fits into the more general context of \emph{manifold calculus}, and we shall need this generalisation at several places.

\subsubsection{Manifold calculus in terms of presheaves}\label{sec:presheaf-model}
Among the various models for the map \eqref{eqn:emb-calc-map} and manifold calculus in general, that of Boavida de Brito and Weiss in terms of presheaves \cite{BoavidaWeissSheaves} is most convenient for our purposes. We refer to Section 8 of their work for a proof of the equivalence between this model and the classical model of \cite{WeissImmersion}.

To recall their model (in a slightly more general setting, see \cref{rem:triad-generalisation}), we fix a $(d-1)$-manifold $K$ possibly with boundary, thought of $\partial_0M$ for manifold triads $M$. We write $\cat{Disc}_{K}$ for the topologically enriched category whose objects are smooth $d$-dimensional manifold triads that are diffeomorphic (as triads) to $K\times [0,1)\sqcup T\times \bR^d$ for a finite set $T$ with $\partial_0(K\times [0,1)\sqcup T\times \bR^d)=K\times \{0\}$, and whose morphisms are given by spaces of embeddings of triads as described in \cref{sect:triads}. If $K$ is clear from the context, we abbreviate $\cat{Disc}_{K}$ by $\cat{Disc}_{\partial_0}$.

We write $\cat{PSh}(\cat{Disc}_{\partial_0})$ for the topologically enriched category of space-valued enriched presheaves on $\cat{Disc}_{\partial_0}$, and we consider it as a category with weak equivalences by declaring a morphism of presheaves to be a weak equivalence if it is a weak equivalence on all its values. Localising at these weak equivalences (for instance as described in \cite{DwyerKan}) gives rise to a topologically enriched category $\cat{PSh}(\cat{Disc}_{\partial_0})^{\loc}$ together with an enriched functor 
\begin{equation}\label{equ:localisation}
	\cat{PSh}(\cat{Disc}_{\partial_0})\lra \cat{PSh}(\cat{Disc}_{\partial_0})^{\loc}.
\end{equation}
Denoting by $\cat{Man}_{\partial_0}$ the topologically enriched category with objects all manifold triads $M$ with an identification $\partial_0M\cong K$ and morphism spaces the spaces of embeddings of triads, a presheaf $F\in \cat{PSh}(\cat{Disc}_{\partial_0})$ induces a new presheaf $T_\infty F\in \cat{PSh}(\cat{Man}_{\partial_0})$ by setting
	\[T_\infty F(M)\coloneqq \Map_{\cat{PSh}(\cat{Disc}_{\partial_0})^{\loc}}\big(\Emb_{\partial_0}(-,M),F \big).\]
If $F$ is the restriction of a presheaf $F\in \cat{PSh}(\cat{Man}_{\partial_0})$, then we have a composition of maps of presheaves 
\begin{equation}\label{equ:manifold-calculus-taylor}
	F(M)\xlra{\cong} \Map_{\cat{PSh}(\cat{Man}_{\partial_0})}\big(\Emb_{\partial_0}(-,M),F \big)\lra T_\infty F(M)
\end{equation}
on $\cat{Man}_{\partial_0}$ where the first map is given by the enriched Yoneda lemma and the second is induced by the restriction along $\cat{Disc}_{\partial_0}\subset \cat{Man}_{\partial_0}$ and the functor \eqref{equ:localisation}. Note that this is a weak equivalence whenever $M\in\cat{Disc}_{\partial_0}$, that is, manifold calculus \emph{converges} on manifolds diffeomorphic to the disjoint union of a collar on $\partial_0M$ and a finite number of open discs.

\begin{example}[Embedding calculus]\label{ex:emb-calc}For triads $M$ and $N$ and a boundary condition $e_{\partial_0}\colon \partial_0M\hookrightarrow \partial_0N$, we have a presheaf $\Emb_{\partial_0}(-,N)$ of embeddings of triads extending $e_{\partial_0}$. Choosing $K=\partial_0M$, the map \eqref{equ:manifold-calculus} gives rise to a model for the embedding calculus map \eqref{eqn:emb-calc-map},
\begin{equation}\label{equ:embedding-calculus-taylor}
\Emb_{\partial_0}(M,N)\lra \Map_{\cat{PSh}(\cat{Disc}_{\partial_0})^{\loc}}\big(\Emb_{\partial_0}(-,M),\Emb_{\partial_0}(-,N) \big)= T_\infty\Emb_{\partial_0}(M,N).
\end{equation}
\end{example}

\begin{remark}\label{rem:alternative-models}There are several alternative points of view on the maps \eqref{equ:manifold-calculus-taylor} and \eqref{equ:embedding-calculus-taylor}, for instance in terms of modules over variants of the little discs operad (see \cite[Section\,6]{BoavidaWeissSheaves} or \cite{Turchin}).
\end{remark}

\subsubsection{A smaller model}\label{sec:smaller-model} In some situations, it is convenient to replace $\cat{Disc}_{\partial_0}$ by a smaller equivalent category. There is a chain of enriched functors
\begin{equation}\label{equ:smaller-cats}
	\cat{Disc}_{\partial_0}^{\bigcdot} \lra \cat{Disc}_{\partial_0}^{\mr{sk}} \lra \cat{Disc}_{\partial_0}.
\end{equation}
The right arrow is the inclusion of the full subcategory $\cat{Disc}_{\partial_0}^{\mr{sk}}\subset \cat{Disc}_{\partial_0}$ on the objects $\partial_0M\times[0,1)\sqcup\underline{n}\times \bR^d$ for $\underline{n}=\{1,\ldots,n\}$ with $n\ge0$. The category $\cat{Disc}_{\partial_0}^{\bigcdot}$ has the same objects as $\cat{Disc}_{\partial_0}^{\mr{sk}}$ and space of morphisms pairs $(s,e)$ of a parameter $s\in(0,1]$ and an embedding of triads 
\[e\colon \partial_0M\times[0,1)\sqcup\underline{n}\times \bR^d\ra \partial_0M\times[0,1)\sqcup\underline{m}\times \bR^d\]
with $e|_{\partial_0M\times[0,1)}=\mr{id}_{\partial_0M}\times {s\cdot(-)}$, where $s\cdot(-)\colon [0,1)\ra [0,1)$ is multiplication by $s$. Composition is given by composing embeddings and multiplying parameters, and the functor to $\cat{Disc}_{\partial_0}^{\mr{sk}}$ forgets the parameters. Both functors in \eqref{equ:smaller-cats} are Dwyer--Kan equivalences, the first by a variant of the proof of the contractibility of the space of collars and the second by definition, so we may equivalently define $T_\infty F(-)$ using any of the three categories \eqref{equ:smaller-cats}.

\subsubsection{Two properties of manifold calculus}\label{sec:formal-properties-mfd} The following two properties of the functor
\begin{equation}\label{equ:manifold-calculus}\cat{PSh}(\cat{Disc}_{\partial_0})\ni F\longmapsto T_\infty F\in \cat{PSh}(\cat{Man}_{\partial_0})\end{equation}
will be of use:

\begin{enumerate}[label=(\alph*),leftmargin=*]
	\item \label{homotopy-limits} \emph{Homotopy limits:} The mapping spaces resulting from the localisation \eqref{equ:localisation} can be viewed equivalently as the derived mapping spaces formed with respect to the projective model structure on $\cat{PSh}(\cat{Disc}_{\partial_0})$ (see \cite[Section 3.1]{BoavidaWeissSheaves}). That is, the functor \eqref{equ:manifold-calculus} models the homotopy right Kan-extension along the inclusion $\cat{Disc}_{\partial_0}\subset \cat{Man}_{\partial_0}$ \cite[Section 4.2]{BoavidaWeissSheaves}. The functor \eqref{equ:manifold-calculus} thus preserves homotopy limits in the projective model structures, which are computed objectwise.
	\item\label{descent} \emph{$\cJ_\infty$-covers and descent:} If $F$ is the restriction of a presheaf $F \in \cat{PSh}(\cat{Man}_{\partial_0})$ then $T_\infty F$ can be seen alternatively as the \emph{homotopy $\cJ_\infty$-sheafification} of $F$: for $1 \leq k \leq \infty$ (we will only use the cases $k=1,\infty$), an open cover $\cU$ of a triad $M$ is called a \emph{Weiss $k$-cover} if every $U \in \cU$ contains an open collar on $\partial_0 M$ and every finite subset of cardinality $\leq k$ of $\mr{int}(M)$ is contained in some element of $\cU$. An enriched presheaf on $\cat{Man}_{\partial_0}$ is a \emph{homotopy $\cJ_k$-sheaf} if it satisfies descent for Weiss $k$-covers in sense of \cite[Definition 2.2]{BoavidaWeissSheaves}. Note that a homotopy $\cJ_1$-sheaf is a homotopy sheaf in the usual sense, and a homotopy $\cJ_k$-sheaf is also a homotopy $\cJ_{k'}$-sheaf for any $k' \ge k$. By \cite[Theorem 1.2]{BoavidaWeissSheaves}, the functor
\[\cat{PSh}(\cat{Man}_{\partial_0})\ni F\longmapsto T_\infty F\in \cat{PSh}(\cat{Man}_{\partial_0})\]
together with the natural transformation $\mr{id}_{\cat{PSh}(\cat{Man}_{\partial_0})}\Rightarrow T_\infty$ is a model for the homotopy $\cJ_\infty$-sheafification. In particular, if $F$ is already a $\cJ_\infty$-sheaf, then $F\ra T_\infty F$ is a weak equivalence, so any map $F\ra G$ in $\cat{PSh}(\cat{Man}_{\partial_0})$ with $G$ a homotopy $\cJ_k$-sheaf for some $1 \leq k\le \infty$ factors over $F\ra T_\infty F$ up to weak equivalence.
		
	It is often convenient to use a stronger version of descent, namely with respect to \emph{complete} Weiss $\infty$-covers $\cU$, which are Weiss $\infty$-covers that contain a Weiss $\infty $-cover of any finite intersection of elements in $\cU$. Regarding $\cU$ as a poset ordered by inclusion, the map induced by restriction
	\[T_\infty F(M) \lra \underset{U \in \cU}\holim\, T_\infty F(U)\]
	 is a weak equivalence by \cite[Lemma 6.7]{KnudsenKupers}.
\end{enumerate}

\begin{remark}\label{rem:weak-map}At several points in the remainder of this work, we will construct maps between spaces of the form $T_\infty \Emb_{\partial_0}(M,N)$ by using the descent property from \cref{sec:formal-properties-mfd} \ref{descent}. Strictly speaking, these  will only be \emph{weak maps} i.e.\,zig-zags of maps whose wrong-way maps are weak equivalences. This will be good enough for all purposes. More formally, a weak map $X\ra Y$ gives an actual morphism from $X$ and $Y$ in the localisation of the category of spaces at the weak equivalences, and all our statements involving weak maps can be viewed as taking place in this localisation. In particular, when we say that a square involving weak maps \emph{commutes up to canonical homotopy} then we mean that the square can be enhanced in a preferred way to a homotopy commutative square in this localisation.
\end{remark}

\subsection{Properties of embedding calculus}\label{sec:formal-properties}
We explain various features of embedding calculus which illustrate that $T_\infty\Emb_{\partial_0}(M,N)$ has formally similar properties to $\Emb_{\partial_0}(M,N)$ even in situations where embedding calculus need not converge.

\begin{enumerate}[label=(\alph*),leftmargin=*]
 	\item\label{postcomposition}\emph{Postcomposition with embeddings:} Given triads $M$, $N$, and $K$, with boundary conditions $e_{\partial_0M}\colon \partial_0M\hookrightarrow \partial_0N$ and $e_{\partial_0N} \colon\partial_0N\hookrightarrow \partial_0K$, there is a map
	 	\[
	T_\infty\Emb_{\partial_0}(M,N)\times \Emb_{\partial_0}(N,K)\lra T_\infty\Emb_{\partial_0}(M,K)
	\]
	that is associative in the evident sense and compatible with the composition maps for embeddings spaces, both up to higher coherent homotopy.
	
	 In the model of \cref{sec:presheaf-model}, these maps are given by applying the map 
	\begin{equation}\label{equ:postcomposition-with-emb}
	 \Emb_{\partial_0}(N,K)\lra \Map_{\cat{PSh}(\cat{Disc}_{\partial_0M})^{\loc}}\big(\Emb_{\partial_0}(-,N),\Emb_{\partial_0}(-,K) \big)
	\end{equation}
	induced by postcomposition in the second factor, followed by composition in $\cat{PSh}(\cat{Disc}_{\partial_0M})^{\loc}$. Note that the codomain of \eqref{equ:postcomposition-with-emb} does in general \emph{not} agree with $T_\infty\Emb_{\partial_0}(N,K)$.

	\item\label{naturality} \emph{Naturality and isotopy invariance:} In the situation of \ref{postcomposition}, if we assume $\dim(M)=\dim(N)$, then there are composition maps
	\begin{equation}\label{equ:composition-emb-calc}
	T_\infty\Emb_{\partial_0}(M,N)\times T_\infty\Emb_{\partial_0}(N,K)\lra T_\infty\Emb_{\partial_0}(M,K)
	\end{equation}
	that are associative in the evident sense and compatible with \eqref{equ:postcomposition-with-emb} and the composition for embeddings, up to higher coherent homotopy. Combining this with \ref{postcomposition}, we see that like spaces of embeddings, $T_\infty\Emb_{\partial_0}(-,-)$ is isotopy-invariant in source and target: if $M \subset M'$ is a sub-triad with $\partial_0M \subset \partial_0M'$ such that there is an embedding of triads $M' \hookrightarrow M$ which is inverse to the inclusion up to isotopy of triads, then the maps
	\[T_\infty \Emb_{\partial_0}(M',N) \ra T_\infty \Emb_{\partial_0}(M,N)\quad\text{and}\quad T_\infty \Emb_{\partial_0}(L,M) \ra T_\infty \Emb_{\partial_0}(L,M')\]
	induced by restriction and inclusion are weak equivalences. Here $L$ is any other triad with a boundary condition $e_{\partial_0}\colon \partial_0L\hookrightarrow\partial_0M$.
	
	In the model described in \cref{sec:presheaf-model}, the composition map \eqref{equ:composition-emb-calc} can implemented as follows: the codimension $0$ embedding $e_{\partial_0M}\colon \partial_0M\hookrightarrow\partial_0N$ induces enriched functor 
	\[(e_{\partial_0M})_*\colon \cat{Disc}^{\bigcdot}_{\partial_0M}\lra \cat{Disc}^{\bigcdot}_{\partial_0N}\quad\text{and}\quad (e_{\partial_0M})^*\colon \cat{PSh}(\cat{Disc}^{\bigcdot}_{\partial_0N})\ra \cat{PSh}(\cat{Disc}^{\bigcdot}_{\partial_0M})\]
	Writing $d\coloneqq \dim(M)=\dim(N)$, the functor $(e_{\partial_0M})_*$ sends on objects $\partial_0M\times[0,1)\sqcup\underline{n} \times \bR^d$ to $\partial_0N\times[0,1)\sqcup\underline{n} \times \bR^d$. On morphisms, $(e_{\partial_0M})_*$ keeps the parameter $s$ fixed and sends an embedding $e$ to the embedding given by $\mr{id}_{\partial_0N}\times (s\cdot (-))$ on $\partial_0M\times[0,1)$ and by $(e_{\partial_0M}\times [0,1)\sqcup \mr{id}_{\underline{n}\times\bR^d})\circ e|_{\underline{n}\times\bR^d}$ on ${\underline{n}\times\bR^d}$. The functor $(e_{\partial_0M})^*$ is given by precomposition with $(e_{\partial_0M})_*$. The restriction maps \[\Emb_{\partial_0N}(\partial_0N\times[0,1)\sqcup\underline{n} \times \bR^d,N)\lra \Emb_{\partial_0M}(\partial_0M\times[0,1)\sqcup\underline{n} \times \bR^d,N)\]
	are weak equivalences by the contractibility of spaces of collars, and similarly for $\Emb_{\partial_0}(-,K)$, so we have weak equivalences in $\cat{PSh}(\cat{Disc}^{\bigcdot}_{\partial_0M})$
	\[\hspace{0.5cm}
	(e_{\partial_0M})^*\Emb_{\partial_0N}(-,N) \xlra{\simeq} \Emb_{\partial_0M}(-,N) \quad\quad (e_{\partial_0M})^*\Emb_{\partial_0N}(-,K) \xlra{\simeq} \Emb_{\partial_0M}(-,K).\]
	Using the model $T_\infty \Emb_{\partial_0}(M,N)\simeq \Map_{\cat{PSh}(\cat{Disc}^{\bigcdot}_{\partial_0M})^{\loc}}\big(\Emb_{\partial_0}(-,M),\Emb_{\partial_0}(-,N) \big)$,
	 the composition \eqref{equ:composition-emb-calc} is given by applying $(e_{\partial_0M})^*$ to the second factor, composition in the category $\cat{PSh}(\cat{Disc}^{\bigcdot}_{\partial_0M})^{\loc}$, and using the weak equivalences of presheaves above.

	\item\label{convergence} \emph{Convergence on disjoint unions of discs:} Embedding calculus converges if the domain $M$ is diffeomorphic (as a triad) to $ \partial_0M\times [0,1)\sqcup T\times \bR^d$ for a finite set $T$, where 
	$\partial_0\big(\partial_0M\times [0,1)\sqcup T\times \bR^d \big)=\partial_0M\times\{0\}$. This follows from the corresponding fact for manifold calculus (see \cref{sec:presheaf-model}). By isotopy invariance, it remains true with $T\times \bR^d$ replaced by $T_1\times \bR^d\sqcup T_2\times D^d$ for finite sets $T_i$.
	
	\item\label{mapping-space} \emph{Comparison to bundle maps:}
	The derivative map $\Emb_{\partial_0}(M,N)\ra \Bun_{\partial_0}(TM,TN)$ fits into a natural commutative diagram (up canonical homotopy) of the form
	 \begin{equation}\label{equ:factorisation-to-maps}
	 \begin{tikzcd} \Emb_{\partial_0}(M,N) \rar\dar &\Bun_{\partial_0}(TM,TN)\rar&\Map_{\partial_0}(M,N) \\
	 	T_\infty \Emb_{\partial_0}(M,N). \arrow[ur,dashed] &  \end{tikzcd}
	\end{equation}
	which is compatible with composition maps from \eqref{equ:composition-emb-calc} up to higher coherent homotopy. This follows from the discussion in \cref{sec:formal-properties-mfd} \ref{descent} by observing that the target in the natural transformation $\Emb_{\partial_0}(-,N) \to \Bun_{\partial_0}(-,TN)$ is a homotopy $\cJ_1$-sheaf, so the map $ \Bun_{\partial_0}(-,TN)\ra  T_\infty\Bun_{\partial_0}(-,TN)$ is a weak equivalence of presheaves.
	
	\item\label{extension-by-identity} \emph{Extension by the identity:} Suppose that we have another triad $Q$ with an identification of $\partial_0 Q$ with a codimension zero submanifold of $\partial_0 M$. Then we can form, up to smoothing corners, the triad $M \cup Q=M\cup_{\partial_0Q}Q$ with $\partial_0(M \cup Q) = (\partial_0 M \setminus \mr{int}(\partial_0 Q)) \cup \partial_1 Q$. If $M$ and $N$ are of the same dimension and we are further given a boundary condition $e_{\partial_0} \colon \partial_0 M \hookrightarrow \partial_0 N$, we can form $N \cup Q$ in the same manner. Extending embeddings by the identity gives a map $\Emb_{\partial_0}(M,N) \ra \Emb_{\partial_0}(M \cup Q,N \cup Q)$ (strictly speaking this requires the addition of collars to the definitions to guarantee the glued map is smooth but we forego the addition of this contractible space of data), which can be shown to fit into a diagram
	\begin{equation}\label{equ:extension-diagram}
		\begin{tikzcd} \Emb_{\partial_0}(M,N) \rar \dar& \Emb_{\partial_0}(M \cup Q,N 	\cup Q) \dar \\
		T_\infty \Emb_{\partial_0}(M,N) \arrow[r,dashed] & T_\infty \Emb_{\partial_0}(M \cup Q,N \cup Q),\end{tikzcd}
	\end{equation}
	commutative up to preferred homotopy.
	The existence of the dashed map in \eqref{equ:extension-diagram} is proved by noting that $T_\infty \Emb_{\partial_0}(- \cup Q,N \cup Q)$ is a homotopy $\cJ_\infty$-sheaf on $\cat{Disc}_{\partial_0 M}$, see \cref{sec:formal-properties-mfd} \ref{descent}.
	\item\label{isotopy-extension} \emph{Isotopy extension:} Suppose that the triads $M$ and $N$ are both $d$-dimensional, and $e_{\partial_0 M} \colon \partial_0 M \hookrightarrow \partial_0 N$ is a boundary condition. Fix a compact $d$-dimensional sub-manifold triad $P\subset M$ (so in particular $\partial_0P=\partial_0 M\cap \partial P$) and consider the induced boundary condition $e_{\partial_0}\colon \partial_0 M \supset \partial_0P \hookrightarrow \partial_0 N$. Suppose that embedding calculus converges for triad embeddings of triads of the form $P \sqcup T\times \bR^d\hookrightarrow N$ for finite sets $T$ in the sense that the map
		\[\Emb_{\partial_0}(P \sqcup T\times \bR^d, N)\lra T_\infty\Emb_{\partial_0}(P \sqcup T\times \bR^d, N)\]
	is a weak equivalence. Then, fixing a triad embedding $e\colon P \hookrightarrow N$ disjoint from $\partial N\backslash e_{\partial_0 M}(\partial_0P)$, there is a map of fibration sequences
	\[\begin{tikzcd}[column sep=0.5cm]\Emb_{\partial_0}\big(M{\setminus}\mr{int}(P),N {\setminus} \mr{int}(e(P))\big) \rar \dar & \Emb_{\partial_0}\big(M,N\big) \dar \rar & \Emb_{\partial_0}\big(P,N\big) \arrow[d,"\simeq"] \\
		T_\infty \Emb_{\partial_0}\big(M {\setminus} \mr{int}(P),N {\setminus} \mr{int}(e(P))\big) \rar & T_\infty \Emb_{\partial_0}\big(M,N\big) \rar &T_\infty \Emb_{\partial_0}\big(P,N\big)
	\end{tikzcd}\]
	whose right square results from \eqref{equ:composition-emb-calc} and whose left square is an instance of the diagram \eqref{equ:extension-diagram}. The homotopy fibres are taken over the embedding $e$ and its image in $T_\infty \Emb_{\partial_0}(P,N)$, and
	$\partial_0(M \setminus \mr{int}(P))\coloneqq\partial_1P\cup \partial_0 M\backslash \mr{int}(\partial_0P)$
	with boundary condition induced by $e$ and $e_{\partial_0 M}$. For the upper row, this is a form of the usual parametrised isotopy extension theorem. For the lower row, this is a mild generalisation of a result of Knudsen and Kupers \cite[Theorem 6.1, Remarks 6.4 and 6.5]{KnudsenKupers}. Note that every triad embedding $P\hookrightarrow N$ is disjoint from $\partial N\backslash e_{\partial_0}(\partial_0P)$ up to isotopy of triad embeddings, so if we would like to draw conclusions about all homotopy fibres of the right horizontal maps, it suffices to restrict to embeddings of this form. 
\end{enumerate}

We record the following immediate corollary of Properties \ref{convergence} and \ref{isotopy-extension} which will allow us to restrict to triads with $\partial_0M\neq \varnothing$ when proving convergence results.
\begin{lemma}\label{lem:remove-discs}
	Let $M$ and $N$ be $d$-dimensional triads, $e_{\partial_0}\colon \partial_0 M\hookrightarrow \partial_0 N$ a boundary condition, and $D^d\subset\mr{int}(M)$ an embedded disc. The map
		\[\Emb_{\partial_0}(M,N)\lra T_\infty\Emb_{\partial_0}(M,N)\] 
	is a weak equivalence if and only if for all embeddings $e\colon D^d\hookrightarrow \mr{int}(N)$ the map
		\[\Emb_{\partial_0}\big(M\backslash \mr{int}(D^d),N\backslash\mr{int}(e(D^d))\big)\lra T_\infty\Emb_{\partial_0}\big(M\backslash \mr{int}(D^d),N\backslash\mr{int}(e(D^d))\big)\]
	is a weak equivalence, where $\partial_0(M\backslash \mr{int}(D^d)) = \partial_0 M \cup \partial D^d$ and $\partial_0(N\backslash \mr{int}(e(D^d))) = \partial_0 N \cup \partial e(D^d)$.
\end{lemma}

\begin{proof}This is an instance of the fact that for a commutative square 
	\[\begin{tikzcd} E \rar \dar & B \dar{\simeq} \\
		E' \rar & B'\end{tikzcd}\]
	whose right arrow is a weak equivalence, the map $E \to E'$ is a weak equivalence if and only if the map $\hofib(E \to B) \to \hofib(E' \to B')$ is a weak equivalence for all choices of basepoints. We apply this to the commutative square induced by restriction
	\[\begin{tikzcd} \Emb_{\partial_0}(M,N) \rar \dar & \Emb(D^d,N) \dar\\
		T_\infty \Emb_{\partial_0}(M,N) \rar & T_\infty \Emb(D^d,N) \end{tikzcd}\]
	whose right-hand map is a weak equivalence by the convergence on discs (Property \ref{convergence}). By isotopy extension (Property \ref{isotopy-extension}),  the map on homotopy fibres over an embedding $e\colon D^d\hookrightarrow \mr{int}(N)$ agrees with the second map in the statement, so the claim follows.
\end{proof}

We continue with a pair of remarks about these properties:

\begin{remark}\label{rem:triad-generalisation}Boavida de Brito and Weiss \cite[Section\,9]{BoavidaWeissSheaves} restrict their attention to the case $\partial_0 M=\partial M$, but this turns out to be no less general: given a manifold triad $M$, the manifold triad $M\backslash \partial_1M$ with $\partial_0(M\backslash \partial_1M)=\mr{int}(\partial_0M)=\partial(M\backslash \partial_1M)$ is isotopy equivalent to $M$, so there is a weak equivalence $T_\infty\Emb_{\partial_0}(M,N) \simeq  T_\infty\Emb_{\partial}(M\backslash\partial_1M,N\backslash\partial_1N)$
by item \ref{naturality} above.
\end{remark}

\begin{remark}\label{rem:dnb} As a consequence of property \ref{mapping-space} above, to show that the map of \cref{acor:main-diff} on path-components $\pi_0\,\Diff_\partial(\Sigma) \to \pi_0\,T_\infty \Emb_\partial(\Sigma,\Sigma)$ is injective, it suffices to prove that
	\begin{equation}\label{eqn:dnb}\pi_0\,\Diff_\partial(\Sigma) \lra \pi_0\,\hAut_\partial(\Sigma)\end{equation}
is injective, which is true for all compact surfaces and can be seen as follows.

First, one reduces to the case of connected surfaces. For this, it suffices to show that closed connected surfaces are homotopy equivalent if and only if they are diffeomorphic, which is a consequence of the fact that closed surfaces are classified by orientability and the Euler characteristic, and both of these are preserved by homotopy equivalences relative to the boundary. In the connected case, the claimed injectivity is proved for instance in \cite[Theorem 4.6]{Boldsen}, with the exception of $\Sigma=S^2$ and $\Sigma=\bR P^2$. These two cases can settled using the fibre sequence resulting from restricting to an embedded $2$-disc and the fact that the mapping class groups of a disc and a Möbius strip are trivial (see \cite[Theorem B]{Smale}, \cite[Theorem 3.4]{Epstein}).

In fact, the forgetful map \eqref{eqn:dnb} is often an isomorphism: for closed orientable surfaces of positive genus this is an instance of the Dehn--Nielsen--Baer theorem \cite[Theorem 8.1]{FarbMargalit}, but there is also an argument for most surfaces with boundary \cite[Theorem 1.1 (1)]{Boldsen}.
\end{remark}

The proof of \cref{athm:main-general} relies on some additional properties of embedding calculus which we establish in the ensuing subsections. These properties are not very surprising, but seem to have not appeared in the literature before.

\subsection{Thickened embeddings} The first property concerns the behaviour of embedding calculus upon replacing the domain $M$ by a thickening, that is, a vector bundle $V$ over $M$. 

Fix manifold triads $M$ and $N$ and a $k$-dimensional vector bundle $p \colon V \to M$. We consider $V$ as a triad via $\partial_0 V\coloneqq p^{-1}(\partial_0M)$. Fixing a boundary condition $e_{\partial_0}\colon \partial_0 V\hookrightarrow \partial_0N$, we obtain a boundary condition $e_{\partial_0}'\colon \partial_0M\hookrightarrow \partial_0N$ by restriction along the zero-section $M\subset V$. From \eqref{equ:factorisation-to-maps}, we obtain the solid arrows in the diagram 
\begin{equation}\label{equ:thicken-embeddings}\begin{tikzcd} \Emb_{\partial_0}(V,N) \rar \dar & \Emb_{\partial_0}(M,N) \dar \\
T_\infty\Emb_{\partial_0}(V,N) \dar \arrow[r,dashed]& T_\infty\Emb_{\partial_0}(M,N) \dar \\
	\mr{Bun}_{\partial_0}(TV,TN) \rar & \mr{Bun}_{\partial_0}(TM,TN).\end{tikzcd}\end{equation}

\begin{lemma}\label{lem:thickened-embeddings}There exists a dashed map in \eqref{equ:thicken-embeddings} such that the diagram commutes up to preferred homotopy and so that the two subsquares are homotopy cartesian.
\end{lemma}

\begin{proof}Let $\cO$ be the poset of open subsets $U \subset M$ containing a collar on $\partial_0 M$. Taking derivatives  as well as restricting embeddings and bundle maps induces a commutative diagram 
	\[\begin{tikzcd}[row sep=0.4cm] \Emb_{\partial_0}(p^{-1}(-),N) \arrow{rr} \arrow{dd} & &  \Emb_{\partial_0}(-,N) \dar\\[-5pt]
	& & T_\infty \Emb_{\partial_0}(-,N) \dar \\[-5pt]
	\Bun_{\partial_0}(Tp^{-1}(-),TN) \rar &  \Bun_{\partial_0}(T-,TN) \rar{\simeq} & T_\infty \Bun_{\partial_0}(T-,TN)\end{tikzcd}\]
of space-valued presheaves on $\cO$, where the bottom equivalence results from the discussion in \cref{sec:formal-properties} \ref{mapping-space}. Since homotopy pullbacks of presheaves are computed objectwise, this is a homotopy-cartesian square of presheaves. We define a new presheaf $F(-)$ on $\cO$ as the homotopy pullback
	\begin{equation}\label{eqn:def-of-f}\begin{tikzcd}F(-) \rar \dar & T_\infty \Emb_{\partial_0}(-,N) \dar \\
	\Bun_{\partial_0}(Tp^{-1}(-),TN) \rar & T_\infty \Bun_{\partial_0}(T-,TN).\end{tikzcd}\end{equation}
The result will follow by evaluation at $M \in \cO$ once we provide an identification \[F(M) \simeq T_\infty \Emb_{\partial_0}(p^{-1}(M),N)=T_\infty\Emb_{\partial_0}(V,N)\] compatible with the maps to $\Bun_{\partial_0}(TV,TN)$ and from $\Emb_{\partial_0}(V,N)$. It follows from \cref{sec:formal-properties-mfd} \ref{descent} and \cref{sec:formal-properties} \ref{convergence}, that it suffices to verify that 
\begin{enumerate}[label=(\alph*)]
	\item $F$ satisfies descent for the complete $J_\infty$-cover $\cU \subset \cO$ given by those open subsets $U \subset M$ equal to a collar on $\partial_0 M$ and a finite collection of open discs, and
	\item the map $\Emb_{\partial_0}(p^{-1}(-),N) \to F(-)$ is a weak equivalence when evaluated on $U \in \cU$.
\end{enumerate}
For (a), we observe that all entries but $F(-)$ in the homotopy pullback diagram \eqref{eqn:def-of-f} defining $F(-)$ satisfy descent with respect to $\cJ_\infty$-covers, so $F(-)$ does as well. For (b), we observe that on $U \in \cU$, the right vertical map of \eqref{eqn:def-of-f} is a weak equivalence so it suffices to verify that $\Emb_{\partial_0}(p^{-1}(U),N) \to \Bun_{\partial_0}(Tp^{-1}(U),TN)$ is a weak equivalence. This is indeed the case because $p^{-1}(U)$ is a disjoint union of a collar on $\partial_0 V$ and a finite collection of open discs.
\end{proof}

We derive from \cref{lem:thickened-embeddings} two lemmas that will allow us to interpolate between convergence questions for $\Emb_{\partial_0}(M,N)$ and for $\Emb_{\partial_0}(V,N)$.

\begin{lemma}\label{lem:thick-arc} Let $M$ and $N$ be manifold triads, $p \colon V \to M$ be a vector bundle considered as a triad by $\partial_0 V = p^{-1}(\partial_0 M)$, and $e_{\partial_0} \colon\partial_0 V \to \partial_0 N$ be a boundary condition. Then the map
	\[\Emb_{\partial_0}(V,N) \lra T_\infty \Emb_{\partial_0}(V,N)\]	
	is a weak equivalence if the map $\Emb_{\partial_0}(M,N) \to T_\infty \Emb_{\partial_0}(M,N)$ is a weak equivalence with boundary condition obtained by restricting $e_{\partial_0}$ to $\partial_0 M \subset \partial_0 V$.\end{lemma}

\begin{proof}This follows from the upper homotopy cartesian square in \eqref{equ:thicken-embeddings} provided by \cref{lem:thickened-embeddings}.
\end{proof}

\begin{lemma}\label{lem:thick-arc-converse} Let $M$ be a $d$-dimensional manifold triad, $N$ be a $(d+k)$-dimensional manifold triad, and $e_{\partial_0}\colon \partial_0 M \hookrightarrow \partial_0 N$ be a boundary condition. Then the map
	\[\Emb_{\partial_0}(M,N) \lra T_\infty \Emb_{\partial_0}(M,N)\]
	is a weak equivalence if the map $\Emb_{\partial_0}(V,N) \to T_\infty \Emb_{\partial_0}(V,N)$ is a weak equivalence for all $k$-dimensional vector bundles $V \to M$ and boundary conditions $\partial_0 V \hookrightarrow \partial N$ extending $e_{\partial_0}$.
\end{lemma}

\begin{proof}
We write $T_\infty \Emb_{\partial_0}(M,N)_\beta$ for the path component of an element $\beta \in T_\infty \Emb_\partial(M,N)$ and $\Emb_{\partial_0}(M,N)_\beta$ for the union of path components mapping to the component of $\beta$. It suffices to prove that $\Emb_{\partial_0}(M,N)_\beta \ra T_\infty \Emb_{\partial_0}(M,N)_\beta$ is a weak equivalence for all $\beta$.
	
	Writing $\beta'\in \Bun_{\partial_0}(TM,TN)$ for the image of $\beta$ under $T_\infty \Emb_{\partial_0}(M,N)\ra \Bun_{\partial_0}(TM,TN)$ from  \cref{sec:formal-properties} \ref{mapping-space}, we choose a metric on $TN$, let $V$ be the vector bundle over $M$ whose fibre over $m \in M$ is the orthogonal complement to $\beta'(T_m M)$ in $T_{\beta'(m)}N$, and extend the boundary condition $e_{\partial_0}\colon \partial_0M\hookrightarrow \partial_0N$  to $\partial_0 V$ by exponentiation. Writing $\Emb_{\partial_0}(V,N)_{\beta}$ and $T_\infty \Emb_{\partial_0}(V,N)_{\beta}$ for the unions of the path components mapping to $\beta$ in \eqref{equ:thicken-embeddings}, \cref{lem:thickened-embeddings} yields a homotopy pullback
	\[\begin{tikzcd} \Emb_{\partial_0}(V,N)_{\beta} \rar \dar{\simeq} & \Emb_{\partial_0}(M,N)_\beta \dar \\
		T_\infty \Emb_{\partial_0}(V,N)_{\beta} \rar & T_\infty \Emb_{\partial_0}(M,N)_{\beta}\end{tikzcd}\]
	whose left vertical map a weak equivalence by assumption. By construction, $\beta'$ lifts to a bundle map in $\Bun_{\partial_0}(TV,TN)$ under the bottom horizontal map in \eqref{equ:thicken-embeddings}, so it follows from \cref{lem:thickened-embeddings} that $T_\infty \Emb_{\partial_0}(V,N)_{\beta}$ is nonempty. As $T_\infty \Emb_{\partial_0}(M,N)_{\beta}$ is path-connected, this implies that the left vertical map in the homotopy pullback is a weak equivalence.
\end{proof}

\subsection{Lifting along covering maps}The second property is concerned with the problem of lifting embeddings of triads $M\hookrightarrow N$ along covering maps $\pi \colon \widetilde{N}\ra N$. To state the result, we consider the cover $\widetilde{N}$ as a triad by setting 
	$\partial_0 \widetilde{N} \coloneqq \pi^{-1}(\partial_0 N)$ and $\partial_1 \widetilde{N} \coloneqq \pi^{-1}(\partial_1 N)$,
and fix a boundary condition $e_{\partial_0}\colon \partial_0 M \hookrightarrow \partial_0 N$ as well as a lift $\tilde{e}_{\partial_0}\colon \partial_0 M \hookrightarrow \partial_0 \widetilde{N}$. We pick a homotopy class $[\alpha] \in \pi_0 \Map_{\partial_0}(M,N)$ such that there exists a lift $[\tilde{\alpha}] \in \pi_0 \Map_{\partial_0}(M,\widetilde{N})$. We shall assume that $\partial_0 M \to M$ is $0$-connected, so that this lift is unique. We write
	\[\Emb_{\partial_0}(M,N)_\alpha\subset \Emb_{\partial_0}(M,N) \quad \text{and} \quad T_\infty \Emb_{\partial_0}(M,N)_\alpha\subset T_\infty \Emb_{\partial_0}(M,N)\]
for the unions of the path components that map to $[\alpha]\in \pi_0\, \Map_{\partial_0}(M,N)$ via the maps in \eqref{equ:factorisation-to-maps}. We similarly define subspaces $\Emb_{\partial_0}(M,\widetilde{N})_{\tilde{\alpha}}\subset \Emb_{\partial_0}(M,\widetilde{N})$ and $T_\infty \Emb_\partial(M,\widetilde{N})_{\tilde{\alpha}}\subset T_\infty \Emb_\partial(M,\widetilde{N})$.

\begin{lemma}\label{lem:lift} In this situation, there exists a dashed map making the diagram
	\[\begin{tikzcd} \Emb_{\partial_0}(M,N)_{\alpha} \rar \dar &  \Emb_{\partial_0}(M,\widetilde{N})_{\tilde{\alpha}} \dar \\
		T_\infty \Emb_{\partial_0}(M,N)_{\alpha} \rar[dashed] & T_\infty \Emb_{\partial_0}(M,\widetilde{N})_{\tilde{\alpha}} \end{tikzcd}\]
	commute up to homotopy.  Here the top map is given by sending an embedding $\beta \in\Emb_\partial(M,N)_\alpha$ to its unique lift $\widetilde{\beta}\in\Emb_\partial(M,\widetilde{N})_{\tilde{\alpha}}$ extending $\tilde{e}_\partial$.
\end{lemma}

\begin{proof}Let $\Emb^\pi_{\partial_0}(-,\widetilde{N}) \subset \Emb_{\partial_0}(-,\widetilde{N})$ be the presheaf on $\cat{Disc}_{\partial_0}$ of those embeddings that remain an embedding after composition with $\pi$. This fits in a pullback diagram
	\[\begin{tikzcd} \Emb^\pi_{\partial_0}(-,\widetilde{N}) \rar{\pi \circ -} \dar &[10pt] \Emb_{\partial_0}(-,N) \dar \\
	\Map_{\partial_0}(-,\widetilde{N}) \rar{\pi \circ -} & \Map_{\partial_0}(-,N) \end{tikzcd}\]
	of presheaves on $\cat{Disc}_{\partial_0M}$ whose vertical maps are given by inclusion. This is homotopy cartesian in the projective model structure on $\cat{PSh}(\cat{Disc}_{\partial_0})$, since $(\pi \circ -) \colon \Map_{\partial_0}(-,\widetilde{N}) \to \Map_{\partial_0}(-,N)$ is a objectwise fibration by the lifting property of covering maps. Evaluating at $M$ and using that $T_\infty(-)$ preserves homotopy limits by \cref{sec:formal-properties-mfd} \ref{homotopy-limits}, we arrive at a commutative cube
		\[\begin{tikzcd} \Emb^\pi_{\partial_0}(M,\widetilde{N}) \arrow{rr} \arrow{dd} \arrow{rd} &[-35pt] &[-35pt] \Emb_{\partial_0}(M,N)  \arrow{dd} \arrow{rd} &[-35pt] \\
		& T_\infty  \Emb^\pi_{\partial_0}(M,\widetilde{N}) \arrow{rr}  & & T_\infty \Emb_{\partial_0}(M,N)  \arrow{dd}  \\
		\Map_{\partial_0}(M,\widetilde{N}) \arrow[two heads]{rr} \arrow{rd}{\simeq} & & \Map_{\partial_0}(M,N) \arrow{rd}{\simeq}  & \\
		& T_\infty \Map_{\partial_0}(M,\widetilde{N}) \arrow[from=uu,crossing over]  \arrow[two heads]{rr} & & T_\infty\Map_{\partial_0}(M,N) \end{tikzcd}
	\]
	with front and back faces homotopy cartesian, and bottom diagonal maps weak equivalences as $\Map_{\partial_0}(-,\widetilde{N})$ and $\Map_{\partial_0}(-,N)$ are homotopy $\cJ_1$-sheaves (see \cref{sec:formal-properties-mfd} \ref{descent}). By the uniqueness of lifts (this uses  that $\partial_0 M \to M$ is $0$-connected), the bottom horizontal maps become weak equivalences when we restrict domain and target to the path components of $[\tilde{\alpha}]$ and $[\alpha]$ respectively. Doing so and using the homotopy pullback property, the top of the cube provides a commutative square		\[
		\begin{tikzcd} \Emb^\pi_{\partial_0}(M,\widetilde{N})_{\tilde{\alpha}} \rar{\simeq} \dar &[10pt] \Emb_{\partial_0}(M,N)_{\alpha} \dar \\
		T_\infty \Emb^\pi_{\partial_0}(M,\widetilde{N})_{\tilde{\alpha}} \rar{\simeq} &[10pt] T_\infty \Emb_{\partial_0}(M,N)_{\alpha}\end{tikzcd}
		\]
		 with horizontal weak equivalences.
The top map is even a homeomorphism, by the uniqueness of lifts. Using the inclusion of presheaves $\Emb^\pi_{\partial_0}(-,\widetilde{N}) \subset \Emb_{\partial_0}(-,\widetilde{N})$, we obtain a commutative diagram
		\[\begin{tikzcd} \Emb_{\partial_0}(M,N)_{\alpha} \dar & \lar[swap]{\cong} \Emb^\pi_{\partial_0}(M,\widetilde{N})_{\tilde{\alpha}} \rar \dar & \Emb_{\partial_0}(M,\widetilde{N})_{\tilde{\alpha}} \dar \\
		T_\infty \Emb_{\partial_0}(M,N)_{\alpha} & \lar[swap]{\simeq} T_\infty \Emb^\pi_{\partial_0}(M,\widetilde{N})_{\tilde{\alpha}} \rar & T_\infty \Emb_{\partial_0}(M,\widetilde{N})_{\tilde{\alpha}}.\end{tikzcd}\]
		whose top composition is given by sending an embedding to its unique lift extending $\tilde{e}_\partial$, so we obtain a map $T_\infty \Emb_{\partial_0}(M,N)_{\alpha}\ra T_\infty \Emb_{\partial_0}(M,\widetilde{N})_{\tilde{\alpha}}$ as desired.
\end{proof}

\begin{remark}If $\alpha$ has no lift, then there is no component of $\Emb_{\partial_0}(M,\widetilde{N})$ mapping to $[\alpha]$ under composition with $\pi$. In this case, the above argument shows that there is also no component of $T_\infty \Emb_{\partial_0}(M,\widetilde{N})$ mapping to $[\alpha]$ under the map of \cref{sec:formal-properties} \ref{mapping-space} and composition with $\pi$.\end{remark}

\subsection{Adding a collar to the source}
The third property concerns the behaviour of embedding calculus when adding a disjoint collar to the domain.

We fix triads $M$ and $N$ and a boundary condition $e_{\partial_0} \colon \partial_0 M \hookrightarrow \partial_0 N$. Given a compact $(\dim(M)-1)$-manifold $K$, we replace $M$ by the triad $M \sqcup K \times [0,1)$ with $\partial_0(M \sqcup K \times [0,1)) = \partial_0 M \sqcup K \times \{0\}$ and fix an extension $e'_{\partial_0} \colon \partial_0(M \sqcup K \times [0,1)) \hookrightarrow \partial N$ of $e_{\partial_0}$ as boundary condition. By contractibility of the space of collars, the restriction map 
	\[\Emb_{\partial_0 M \sqcup K \times \{0\}}(M \sqcup K \times [0,1),N) \lra \Emb_{\partial_0}(M,N)\]
is a weak equivalence. Embedding calculus has this property as well:

\begin{lemma}\label{lem:forget-collar} In this situation, both horizontal maps in the diagram induced by restriction
	\[\begin{tikzcd} \Emb_{\partial_0 M \sqcup K \times \{0\}}(M \sqcup K \times [0,1),N) \rar{\simeq} \dar &  \Emb_{\partial_0M}(M,N) \dar \\
		T_\infty \Emb_{\partial_0 M \sqcup K \times \{0\}}(M \sqcup K \times [0,1),N) \rar{\simeq} & T_\infty \Emb_{\partial_0M}(M,N),\end{tikzcd}\]
	 are weak equivalences.
\end{lemma}

\begin{proof} Let $\cU$ be the open cover of $M \sqcup K \times [0,1)$ given by subsets of the form $U = V \sqcup K \times [0,1)$ where $V \subset M$ is the union of a open subset diffeomorphic to a collar on $\partial_0 M$ and a finite disjoint union of open discs. This is a complete Weiss $\infty$-cover of $M \sqcup K \times [0,1)$, and $\cU' = \{U \cap M \mid U \in \cU\}$ is a complete Weiss $\infty$-cover of $M$. Restriction thus induces a commutative diagram
\[\begin{tikzcd} \Emb_{\partial_0 M \sqcup K \times \{0\}}(M \sqcup K \times [0,1),N) \rar \dar & \holim_{U \in \cU} \Emb_{\partial_0 M \sqcup K \times \{0\}}(U,N) \dar{\simeq} \\
	T_\infty \Emb_{\partial_0 M \sqcup K \times \{0\}}(M \sqcup K \times [0,1),N) \rar{\simeq} & \holim_{U \in \cU} T_\infty \Emb_{\partial_0 M \sqcup K \times \{0\}}(U,N)\end{tikzcd}\]	
	whose bottom horizontal map is a weak equivalences by \cref{sec:formal-properties-mfd} \ref{descent} and whose right vertical map is a weak equivalence by \cref{sec:formal-properties} \ref{convergence}. Similarly we have a square
	\[\begin{tikzcd} \Emb_{\partial_0M}(M ,N) \rar \dar & \holim_{U \in \cU} \Emb_{\partial_0M}(U \cap M,N) \dar{\simeq} \\
		T_\infty \Emb_{\partial_0M}(M,N) \rar{\simeq} & \holim_{U \in \cU} T_\infty \Emb_{\partial_0M}(U \cap M,N),\end{tikzcd}\]	
	which receives a map from the former square by restriction, so it suffices to show that the maps
	\[\Emb_{\partial_0 M \sqcup K \times \{0\}}(U,N) \lra \Emb_{\partial_0M}(U \cap M,N)\]
	are weak equivalence. This follows from the contractibility of spaces of collars.
\end{proof}

Combined with \cref{lem:remove-discs} this yields the following lemma, which is often useful to justify the hypothesis needed to apply isotopy extension for embedding calculus (see \cref{sec:formal-properties} \ref{isotopy-extension}).

\begin{lemma}\label{lem:thick-arc-discs}Let $M$ and $N$ be $d$-dimensional triads, and $e_{\partial_0}\colon \partial_0 M\hookrightarrow \partial_0 N$ a boundary condition. Then the map 
	\[\Emb_{\partial_0}(M \sqcup (T \times \bR^d),N) \lra T_\infty \Emb_{\partial_0}(M\sqcup (T \times \bR^d),N)\]
	is a weak equivalence for any finite set $T$, if the maps $\Emb_{\partial_0}(M,N') \to T_\infty \Emb_{\partial_0}(M,N')$ are weak equivalences for all $d$-dimensional triads $N'$ and all boundary conditions $e'_{\partial_0}\colon \partial_0 M \hookrightarrow \partial_0 N'$.
\end{lemma}

\begin{proof}By induction over $|T|$ it suffices to prove the case $|T|=1$. In that case, it suffices by \cref{lem:remove-discs} to prove that for all embeddings $e \colon D^d \hookrightarrow \mr{int}(N)$ the map
	\[\Emb_\partial\big(M \sqcup (\bR^d{\setminus} \mr{int}(D^d)),N{\setminus}\mr{int}(e(D^d))\big) \lra T_\infty \Emb_\partial\big(M\sqcup (\bR^d{\setminus} \mr{int}(D^d)),N{\setminus}\mr{int}(e(D^d)\big)\]
is a weak equivalences. By  \cref{lem:forget-collar} we may then forget the collars $(\bR^d{\setminus} \mr{int}(D^d))$ on $\partial D^d$ from the source, so the result follows.
\end{proof}

\subsection{Taking disjoint unions}The fourth and final general property of embedding calculus we shall discuss concerns taking disjoint unions in source and target. Its full strength is not needed to prove the main results of this paper---only \cref{cor:add-components} is---but we believe it to be of independent interest. 

Let $M$, $M'$, $N$, and $N'$ be triads with $\dim(M) = \dim(M')$ and $\dim(N) = \dim(N')$. Given boundary conditions $e_{\partial_0} \colon \partial_0 M \hookrightarrow \partial_0 N$ and $e'_{\partial_0} \colon \partial_0 M' \hookrightarrow \partial_0 N'$, we consider the boundary condition $e_{\partial_0} \sqcup e'_{\partial_0} \colon \partial_0(M \sqcup M') \hookrightarrow \partial_0(N \sqcup N')$. Disjoint union of embeddings induces 
\[\Emb_{\partial_0}(M,N) \times \Emb_{\partial_0}(M',N') \lra \Emb_{\partial_0}(M \sqcup M',N \sqcup N')\]
which is a weak equivalence (in fact, a homeomorphism) if both inclusions $\partial_0 M \hookrightarrow M$ and $\partial_0 M' \hookrightarrow M'$ are $0$-connected. Embedding calculus has this property as well:

\begin{lemma}\label{lem:disjoint-unions} In this situation, there is a dashed weak equivalence that makes
	\[\begin{tikzcd} \Emb_{\partial_0}(M,N) \times \Emb_{\partial_0}(M',N') \rar{\simeq} \dar &  \Emb_{\partial_0}(M \sqcup M',N \sqcup N') \dar \\
		T_\infty \Emb_{\partial_0}(M,N) \times T_\infty \Emb_{\partial_0}(M',N') \rar[dashed]{\simeq} &  T_\infty \Emb_{\partial_0}(M \sqcup M',N \sqcup N') \end{tikzcd}\]
		commute up to preferred homotopy.
\end{lemma}

\begin{proof}As in the proof of \cref{lem:forget-collar}, the property of embedding calculus we shall use is descent for complete Weiss $\infty$-covers (see \cref{sec:formal-properties-mfd} \ref{descent}). 
	
We take $\cU_M$ to be the open cover of $M$ given by open subsets $U \subset M$ that are diffeomorphic to a collar on $\partial_0 M$ and a finite disjoint union of open discs, and similarly for $\cU_{M'}$. We take $\cU_{M \sqcup M'}$ to be the open cover of $M \sqcup M'$ given by unions of an element of $\cU_M$ and an element of $\cU_{M'}$. The covers $\cU_M$, $\cU_{M'}$, and $\cU_{M\sqcup M'}$ are all complete Weiss $\infty$-covers. 

We consider $\cU_{M \sqcup M'}$ as a poset ordered by inclusion and let $\Emb^\sqcup_{\partial_0}(-,N \sqcup N')$ be the presheaf on $\cU_{M \sqcup M'}$ that sends $U \sqcup U'$ with $U \in \cU_M$ and $U' \in \cU_{M'}$ to the subspace $\Emb^\sqcup_{\partial_0}(U \sqcup U',N \sqcup N') \subset \Emb_{\partial_0}(U \sqcup U',N \sqcup N')$ which map $U$ into $N$ and $U'$ into $N'$. Defining $\Map^\sqcup_{\partial_0}(-,N \sqcup N')$ similarly, we have a homotopy pullback diagram of presheaves on $\cU_{M \sqcup M'}$
	\begin{equation}\label{eqn:disjoint-union-pb}\begin{tikzcd}\Emb^\sqcup_{\partial_0}(-,N \sqcup N') \rar \dar & \Emb_{\partial_0}(-,N \sqcup N') \dar\\
	\Map^\sqcup_{\partial_0}(-,N \sqcup N') \rar & \Map_{\partial_0}(-,N \sqcup N'),\end{tikzcd}\end{equation}
	and this remains a homotopy pullback when taking homotopy limits over $\cU_{M \sqcup M'}$.

To identify the term
\[\underset{U \sqcup U' \in \cU_{M \sqcup M'}}{\holim} \Emb^\sqcup_{\partial_0}(U \sqcup U',N \sqcup N')\]
we note that there are isomorphisms $\cU_{M \sqcup M'} \cong \cU_M \times \cU_{M'}$ of categories, and $\Emb^\sqcup_{\partial_0}(-,N \sqcup N')\cong \Emb_{\partial_0}(-,N)\times \Emb_{\partial_0}(-, N')$ of presheaves, so the Fubini theorem for homotopy limits implies that this homotopy limit is given by
\[\underset{U \in \cU_M}{\holim}\, \Emb_{\partial_0}(U,N) \times \underset{U' \in \cU_M}{\holim}\, \Emb_{\partial_0}(U',N').\]
Combining descent with the fact that embedding calculus converges on $U \in \cU_M$ and $U' \in \cU_{M'}$ by \cref{sec:formal-properties} \ref{convergence}, we conclude that
\[\underset{U \sqcup U' \in \cU_{M \sqcup M'}}{\holim} \Emb^\sqcup_{\partial_0}(U \sqcup U',N \sqcup N') \simeq T_\infty \Emb_{\partial_0}(M,N) \times T_\infty \Emb_{\partial_0}(M',N').\]
The same analysis holds for $\Map^\sqcup_{\partial_0}(-,M \sqcup M')$ and since this is a homotopy $\cJ_1$-sheaf (see \cref{sec:formal-properties-mfd} \ref{descent}), we conclude that
\[\underset{U \sqcup U' \in \cU_{M \sqcup M'}}{\holim} \Map^\sqcup_{\partial_0}(U \sqcup U',N \sqcup N') \simeq \Map_{\partial_0}(M,N) \times \Map_{\partial_0}(M',N').\]
By the same argument (using descent, convergence on $U \sqcup U' \in \cU_{M \sqcup M'}$, and that $\Map_{\partial_0}(-,N \sqcup N')$ is a homotopy $\cJ_1$-sheaf), we have weak equivalences
\begin{align*}
\holim_{U \sqcup U' \in \cU_{M \sqcup M'}}\Emb_{\partial_0}(U \sqcup U',N \sqcup N')&\simeq T_\infty\Emb_{\partial_0}(M\sqcup M',N\sqcup N') \quad \text{and}\\
\holim_{U \sqcup U' \in \cU_{M \sqcup M'}}\Map_{\partial_0}(U \sqcup U',N \sqcup N')&\simeq \Map_{\partial_0}(M\sqcup M',N\sqcup N'),\end{align*}
so altogether we obtain a homotopy pullback diagram of the form
\[\begin{tikzcd} T_\infty \Emb_{\partial_0}(M,N) \times T_\infty \Emb_{\partial_0}(M',N') \rar \dar & T_\infty \Emb_{\partial_0}(M \sqcup M',N \sqcup N') \dar \\
\Map_{\partial_0}(M,N) \times \Map_{\partial_0}(M',N') \rar & \Map_{\partial_0}(M \sqcup M',N \sqcup N'). \end{tikzcd}\]
The condition that $\partial_0 M \hookrightarrow M$ and $\partial_0 M' \hookrightarrow M'$ are  $0$-connected implies that the bottom map is a weak equivalence, so the top map is a weak equivalence as well. The proof is finished by tracing through the weak equivalences to see that this makes the square in the statement homotopy commute.
\end{proof}

Taking $M' = \varnothing$, which is the only case used in this paper, \cref{lem:disjoint-unions} says:

\begin{corollary}\label{cor:add-components} In this situation, in the diagram induced by the inclusion $N\hookrightarrow N\sqcup N'$
	\[\begin{tikzcd} \Emb_{\partial_0}(M,N) \rar{\simeq} \dar &  \Emb_{\partial_0}(M,N \sqcup N') \dar \\
		T_\infty \Emb_{\partial_0}(M,N) \rar{\simeq} & T_\infty \Emb_{\partial_0}(M,N \sqcup N'), \end{tikzcd}\]
	both horizontal maps are weak equivalences.
\end{corollary}

\begin{remark}\cref{cor:add-components} admits an alternative proof along the lines of \cref{lem:lift}: one observes there is a homotopy pullback diagram of presheaves on $\cat{Disc}_{\partial_0 M}$ given by 
	\[\begin{tikzcd} \Emb_{\partial_0}(-,N) \rar \dar & \Emb_{\partial_0}(-,N \sqcup N') \dar \\
		\Map_{\partial_0}(-,N) \rar & \Map_{\partial_0}(-,N \sqcup N').\end{tikzcd}\]
Taking $T_\infty$ and evaluating at $M$ yields a homotopy pullback diagram of spaces and if $\partial_0 M \to M$ is $0$-connected, the map $\Map_{\partial_0}(M,N) \to \Map_{\partial_0}(M,N \sqcup N')$ is a weak equivalence and hence so is the map $T_\infty \Emb_{\partial_0}(M,N) \to T_\infty \Emb_{\partial_0}(M, N \sqcup N')$.
\end{remark}

\section{Convergence in low dimensions}
In this section we make use the properties of embedding calculus discussed in the previous section to prove the following convergence result. \cref{athm:main-general} is included as the special case $\partial_0M=\partial M$.
\begin{theorem}
\label{thm:main-general-triads} 
	For compact manifolds triads $M$ and $N$ with $\dim(N)\le 2$, the map
		\[\Emb_{\partial_0}(M,N) \lra \ T_\infty \Emb_{\partial_0}(M,N)\]
	is a weak equivalence for any boundary condition $e_{\partial_0}\colon \partial_0 M\hookrightarrow \partial_0 N$.
\end{theorem}

\begin{convention}\label{triad-convention}Throughout this section, we adopt following conventions on triads:
	\begin{enumerate}[leftmargin=*, label=(\roman*)]
		\item We write $I=[0,1]$ and call the $1$-dimensional triads $I$ and $I\times[0,1]$ with $\partial_0I=\{0,1\}$ and $\partial_0(I\times[0,1])=\{0,1\}\times [0,1]$ the \emph{arc} and the \emph{strip}. We will use the convention and notation from \cref{sect:triads}, so embeddings $I\times [0,1]$ into a triad $N$ will always be assumed to extend a boundary condition $e_{\partial_0}\colon \{0,1\}\times [0,1]\hookrightarrow \partial_0N$ which will either be specified or is clear from the context. We consider $I$ as a submanifold of $I\times [0,1]$ via the inclusion $\{1/2\}\times [0,1]\subset I\times [0,1]$, so a boundary condition $e_{\partial_0}$ as above in particular induces a boundary condition $e_{\partial}\colon \{0,1\}\hookrightarrow N$ for embedding of the form $I\hookrightarrow N$ by restriction.
		\item We consider the \emph{cylinder} $S^1 \times [0,1]$ as a manifold triad with $\partial_0(S^1 \times [0,1]) = \varnothing$. We consider the \emph{circle} $S^1$ as the submanifold of $S^1 \times [0,1]$ via the inclusion $S^1 \times \{1/2\} \hookrightarrow S^1 \times [0,1]$.
		\item We consider the \emph{Möbius strip} $\mr{Mo} = ([0,1] \times [0,1])/{\sim}$, with $\sim$ the equivalence relation generated by $(0,y) \sim (1,1-y)$, as a manifold triad with $\partial_0(\mr{Mo}) = \varnothing$. We consider $S^1$ as the submanifold of $\mr{Mo}$ via the inclusion $S^1 \times \{1/2\} \hookrightarrow \mr{Mo}$. 
		\item We write $\Sigma_{g,n}$ for an orientable compact surface of genus $g$ with $n$ boundary components, considered as a manifold triad with $\partial_0\Sigma_{g,n}=\partial \Sigma_{g,n}$.
	\end{enumerate}
\end{convention}

\subsection*{The steps}The proof of \cref{thm:main-general-triads} is divided into the following steps. 
\begin{enumerate}[label=(\arabic*)]
\item \label{enum:silly} $\dim(M)>\dim(N)$ or $\dim(M)=0$,
\item  \label{enum:surfaces}$\dim(M)\le \dim(N)=2$, with substeps:
\begin{enumerate}[label=(\arabic{enumi}.\arabic*)]
\item \label{enum:step-interval}$M$ an arc or a strip,
\item \label{enum:step-circle}$M$ a circle, a cylinder, or a Möbius band,
\item \label{enum:step-line-over-1dim}$M$ a line bundle over a  $1$-dimensional triad $M'$ with $\partial_0M'=\partial M$,
\item \label{enum:step-1dim}$M$ a general $1$-dimensional triad,
\item \label{enum:step-disc}$M=D^2$ with $\partial_0M=\partial M$,
\item \label{enum:step-genus0}$M$ an orientable genus $0$ surface with $\partial_0M=\partial M$,
\item\label{enum:step-connected-full-bdy} $M$ a connected $2$-dimensional triad with $\partial_0M=\partial M$,
\item \label{enum:step-connected}$M$ a connected $2$-dimensional triad with $\partial_0M\neq \partial M$,
\item \label{enum:step-general}$M$ a general $2$-dimensional triad,
\end{enumerate}
\item\label{enum:step-both-1dim}$\dim(M)=\dim(N)=1$.
\end{enumerate}

\noindent Too avoid being repetitive, we say that \emph{convergence holds for a pair of triads $(M,N)$} if the map
	\[\Emb_{\partial_0}(M,N)\lra T_\infty\Emb_{\partial_0}(M,N)\]
is a weak equivalence for all boundary conditions $e_{\partial_0}\colon \partial_0M\hookrightarrow \partial_0N$.

\subsection*{Step \ref{enum:silly}}\textit{Convergence holds for $(M,N)$ if $\dim(M)>\dim(N)$ or $\dim(M)=0$.}

\medskip
\noindent Convergence for $\dim(M)=0$ holds as a result of \cref{sec:formal-properties} \ref{convergence}. For $M\neq\varnothing$ and $\dim(M)>\dim(N)$, we consider the composition
$\Emb_{\partial_0}(M,N)\ra T_\infty\Emb_{\partial_0}(M,N)\ra\Bun_{\partial_0}(TM,TN)$ from \cref{sec:formal-properties} \ref{mapping-space}. If $\dim(M)>\dim(N)$ then the final space in this composition is empty, so the same holds for the first and the second space. This implies convergence.

\subsection*{Step \ref{enum:step-interval}} \textit{Convergence holds for $(M,N)$ if $M$ is an arc or a strip, and $\dim(N)=2$}

\medskip

\noindent We divide this step into two substeps: the case where the boundary condition $e_{\partial_0}\colon\partial_0M\hookrightarrow \partial_0N$ hits two distinct boundary components of $N$, and the case where the boundary condition hits a single boundary component. The arguments are inspired by Gramain's work \cite{Gramain} and Hatcher's exposition thereof in \cite{HatcherMW}. 

\subsubsection*{Substep: The boundary condition hits two distinct boundary components of $N$}By \cref{lem:thick-arc-converse} and isotopy invariance (see \cref{sec:formal-properties} \ref{naturality}), it suffices to consider the case $M=I\times[0,1]$ of a strip. To do so, we glue a disc $D$ to the boundary component of $N$ hit by $\{1\}$, and consider $L = (I \times [0,1] \cup D)$. Smoothing corners and an application of isotopy extension justified by the convergence on discs (see \cref{sec:formal-properties} \ref{convergence} and \ref{isotopy-extension}) yields a map of fibre sequence
	\[\begin{tikzcd} \Emb_{\partial_0}(I \times [0,1],N) \dar  \rar & \Emb_{I \times \{0\}}(L,N \cup D) \rar  \dar & \Emb(D,N \cup D) \dar{\simeq} \\
		T_\infty \Emb_{\partial_0}(I \times [0,1],N) \rar & T_\infty \Emb_{I \times \{0\}}(L,N \cup D) \rar & T_\infty\Emb(D,\,N \cup D)\end{tikzcd}\]
	with fibres taken over the standard inclusion. Since $L$ is isotopy equivalent to $I \times [0,1)$ relative to $I\times \{0\}$, the middle vertical map is a weak equivalence by isotopy invariance and the convergence on collars (see \cref{sec:formal-properties} \ref{naturality} and \ref{convergence}), so the left vertical map is a weak equivalence as well.

\subsubsection*{Substep: The boundary condition hits a single boundary components of $N$}
The case of arcs and strips connecting the same boundary component is harder and its proof is the heart of the overall argument. It relies on \cref{lem:lift} on lifting embeddings, which we spell out again in the special case we shall use. 

This involves a covering map $\widetilde{N}\ra N$, a boundary condition $e_\partial\colon \{0,1\} \hookrightarrow \partial N$, a path $\alpha\in\Map_\partial(I,N)$, and a lift $\tilde{\alpha} \colon I \to \smash{\widetilde{N}}$ of $\alpha$ whose endpoints induce a boundary condition $e_\partial\colon \{0,1\} \hookrightarrow \partial \widetilde{N}$. Recall that $\Emb_\partial(I,N)_\alpha\subset \Emb_\partial(I,N)$ and $T_\infty \Emb_\partial(I,N)_\alpha\subset T_\infty \Emb_\partial(I,N)$
denote the collections of path components that map to $[\alpha]\in \pi_0 \Map_\partial(I,N)$ via the maps in \eqref{equ:factorisation-to-maps}. \cref{lem:lift} for the triad $M=I$ with $\partial_0I=\{0,1\}$ then gives:

\begin{lemma}\label{lem:lift-arcs} In this situation, there exists a dashed map making the diagram
	\[\begin{tikzcd} \Emb_\partial(I,N)_\alpha \rar \dar & \Emb_\partial(I,\widetilde{N})_{\tilde{\alpha}} \dar \\
		T_\infty \Emb_\partial(I,N)_\alpha \arrow[r,dashed] & T_\infty \Emb_\partial(I,\widetilde{N})_{\tilde{\alpha}}\end{tikzcd}\]
	commute up to homotopy. Here the top map is given by sending an arc $\gamma\in\Emb_\partial(I,N)_\alpha$ to the unique lift $\widetilde{\gamma}\in\Emb_\partial(I,\widetilde{N})_{\tilde{\alpha}}$ starting at $\tilde{\alpha}(0)\in\widetilde{N}$.
\end{lemma}

Using this lemma, we now prove convergence for $(M,N)$ if $M$ is an arc or a strip, $\dim(N)=2$, and the boundary condition $e_{\partial_0}\colon \partial_0M\hookrightarrow \partial_0N$ hits a single boundary components of $N$.

\medskip

\noindent By \cref{lem:thick-arc} and isotopy invariance (see \cref{sec:formal-properties} \ref{naturality}) it suffices to prove the claim for the arc, and by \cref{cor:add-components}, we may assume that the target $N$ is connected. We attach a 1-handle $I \times [0,1]$ to $N$ to the boundary component hit by $\{0,1\}$, such that $I \times \{0\}$ and $I \times \{1\}$ are separated on that boundary component by $\{0,1\}$ and are embedded with opposite orientation, resulting in a new surface $P$ with an additional boundary component; see Figure \ref{fig:gamma}. The composition $\{0,1\}\hookrightarrow N\subset P$ now hits two distinct boundary components, so the right vertical map in the homotopy-commutative diagram induced by the inclusion $N\subset P$ (see  \cref{sec:formal-properties} \ref{postcomposition})
	\begin{equation}\label{equ:extension-to-gamma}\begin{tikzcd} \Emb_\partial(I,N) \rar \dar & \Emb_\partial(I,P) \dar{\simeq} \\
			T_\infty \Emb_\partial(I,N) \rar & T_\infty \Emb_\partial(I,P) \end{tikzcd}\end{equation}
	is a weak equivalence by the previous substep. 
	
	\begin{figure}
		\begin{tikzpicture}[scale=0.8]
			\draw [dotted] (0,0) circle (2cm);
			\draw [domain=160:20] plot ({2*cos(\x)}, {2*sin(\x)});
			\draw [domain=-160:-20] plot ({2*cos(\x)}, {2*sin(\x)});
			\draw ({2*cos(20)},{2*sin(20)}) to[out=0,in=0,looseness=2] (0,3);
			\draw ({2*cos(160)},{2*sin(160)}) to[out=180,in=180,looseness=2] (0,3);
			\draw ({2*cos(-20)},{2*sin(-20)}) to[out=0,in=0,looseness=2] (0,4);
			\draw ({2*cos(-160)},{2*sin(-160)}) to[out=180,in=180,looseness=2] (0,4);
			\draw [ultra thick,Mahogany] (0,-2) -- (0,2);
			\node [right] at (0,0) {$\alpha$};
			\draw [very thick,Mahogany,dashed] (0,3) -- (0,4);
			\node [right] at (0,3.5) {$\beta$};
			
			\draw[yshift=-.3cm] (-1.5,-.5) to[out=0,in=-90] (-1.4,0.8) to[out=90,in=180] (-1,1.2) to[out=0,in=90] (-.6,0.8) to[out=-90,in=180] (-.5,-.5);	
			\draw[yshift=-.3cm] (-1.2,0.8) to[out=-90,in=-90] (-0.8,0.8);
			\draw[yshift=-.3cm] (-1.15,0.75) to[out=90,in=90] (-0.85,0.75);
			
			\draw[xshift=2cm,yshift=.1cm] (-1.5,-.5) to[out=0,in=-90] (-1.4,0.8) to[out=90,in=180] (-1,1.2) to[out=0,in=90] (-.6,0.8) to[out=-90,in=180] (-.5,-.5);	
			\draw[xshift=2cm,yshift=.1cm] (-1.2,0.8) to[out=-90,in=-90] (-0.8,0.8);
			\draw[xshift=2cm,yshift=.1cm] (-1.15,0.75) to[out=90,in=90] (-0.85,0.75);
			
			\node at (1,-1) {$N$};
			\node at (-2,3) {$P$};
		\end{tikzpicture}
		\caption{The surface $P$. The original surface $N$ is the region within the dotted circle.}
		\label{fig:gamma}
	\end{figure}
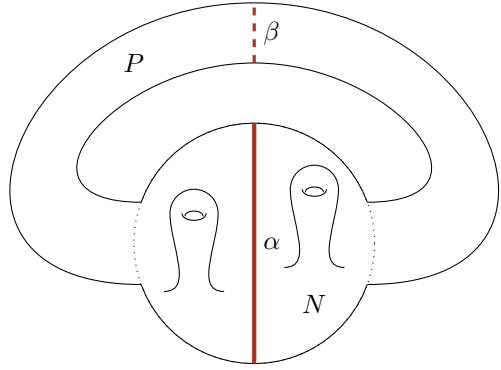
	
	We next investigate the set of path-components. To do so, we will use that the dashed map in 
		\[\begin{tikzcd}
		\pi_0\, \Emb_\partial(I,N)\arrow[dr,dashed]\arrow[rrd, bend left=8]\arrow[ddr, bend right=20]&&\\[-10pt]
		&\big(\pi_0\, \Map_\partial(I,N)\big)\times_{\pi_0\, \Map_\partial(I,P)}\big(\pi_0\, \Emb_\partial(I,P)\big) \rar \dar & \pi_0\, \Emb_\partial(I,P) \dar \\
		&\pi_0\, \Map_\partial(I,N) \rar & \pi_0\, \Map_\partial(I,P),\end{tikzcd}\]
	is surjective: if an embedding $I \hookrightarrow P$ is homotopic to a map $I \to N$, then it is isotopic to an embedding $I \hookrightarrow N$ within the homotopy class of $I\to N$. To see this, use the bigon criterion \cite[Sections 1.2.4, 1.2.7]{FarbMargalit} to isotope $I \hookrightarrow P$ so that its geometric intersection number with the cocore $\beta$ of the 1-handle is equal to the algebraic intersection number, which is $0$ since it is homotopic to a map $I \to N$. With this in mind, a diagram chase in the factorisation
	\[\begin{tikzcd} \pi_0\, \Emb_\partial(I,N) \arrow{r} \arrow[dd, bend right=90, swap, "\circled{1}"] \arrow{d} & \pi_0\, \Emb_\partial(I,P) \arrow[d,"\cong"] \\
		\pi_0\, T_\infty \Emb_\partial(I,N) \arrow{r} \arrow[d,"\circled{2}",swap] &\pi_0\, T_\infty \Emb_\partial(I,P)\arrow{d} \\	
		\pi_0\, \Map_\partial(I,N) \arrow{r} &  \pi_0\, \Map_\partial(I,P)\end{tikzcd}\]
	shows that the maps \circled{1} and \circled{2} have the same image.
	
	Let us now fix a class $[\alpha] \in \pi_0\, \Map_\partial(I,N)$ in this image. As the map \circled{1} is injective because two embedded arcs are isotopic relative to the endpoints if and only if they are homotopic relative to the endpoints (see \cite{Feustel}), there is a unique path component $\Emb_\partial(I,N)_\alpha$ of $\Emb_\partial(I,N)$ mapping to $[\alpha]$.  Denoting by $T_\infty \Emb_\partial(I,N)_\alpha\subset T_\infty \Emb_\partial(I,N)$ the union of all path components that map to $[\alpha]$, it suffices to show that the map
	$\Emb_\partial(I,N)_\alpha \ra T_\infty \Emb_\partial(I,N)_\alpha$
	is a weak equivalence for all choices of $[\alpha]$. Since $\Emb_\partial(I,N)_\alpha$ is contractible by \cite[Th\'eor\`eme 5]{Gramain}, the task is to prove that $T_\infty \Emb_\partial(I,N)_\alpha$ is (weakly) contractible as well. 
	
	To do so, we will construct a homotopy-commutative diagram
	\begin{equation}\label{eqn:retraction-diagram}\begin{tikzcd} \Emb_\partial(I,N)_\alpha \rar \dar & \Emb_\partial(I,P)_{{\alpha}} \rar{(e \circ -) \circ \text{lift}} \arrow[d,"\simeq"] &[20pt]  \Emb_\partial(I,N)_\alpha \dar \\
			T_\infty \Emb_\partial(I,N)_\alpha \rar & T_\infty \Emb_\partial(I,P)_{{\alpha}} \rar{(e \circ -) \circ \text{lift}} &  T_\infty \Emb_\partial(I,N)_\alpha\end{tikzcd}\end{equation}
	whose horizontal compositions are homotopic to the identity. This will finish the proof, since it exhibits $T_\infty \Emb_\partial(I,N)_\alpha$ as a retract of the contractible space $T_\infty \Emb_\partial(I,P)_{{\alpha}} \simeq \Emb_\partial(I,P)_{{\alpha}}$.
	
	The left square in \eqref{eqn:retraction-diagram} is obtained by restricting the path-components of the homotopy commutative square \eqref{equ:extension-to-gamma}. The right square arises as the composition of two squares 
	\[\begin{tikzcd}\Emb_\partial(I,P)_\alpha \rar{\text{lift}} \dar{\simeq} & \Emb_\partial(I,\widetilde{P})_{\tilde{\alpha}} \rar{e \circ -} \dar & \Emb_\partial(I,N)_{\alpha} \dar \\
		T_\infty \Emb_\partial(I,P)_\alpha \rar{\text{lift}} & T_\infty \Emb_\partial(I,\widetilde{P})_{\tilde{\alpha}} \rar{e \circ -} & T_\infty\Emb_\partial(I,N)_{\alpha}. \end{tikzcd}\]
	which we explain now. The surface $\widetilde{P}$ is an appropriate covering space of $P$: the construction of $P$ gives a decomposition $\pi_1(P) \cong \pi_1(N) \ast \bZ$ and $\widetilde{P}$ is the cover corresponding to the subgroup $\pi_1(N)$. Explicitly, the cover $\widetilde{P}$ can be constructed by cutting $P$ along $\beta$ to obtain a surface $R$ (see Figure \ref{fig:tildegamma}) and gluing two copies of the universal cover $\widetilde{R}$ of this surface to the two dashed intervals in the boundary resulting from $\beta$. Note that $R$ contains a preferred lift $\tilde{\alpha}$ of $\alpha$ and hence so does $\widetilde{P}$. We denote the endpoints of $\alpha$ and $\widetilde{\alpha}$ in the various surfaces generically by $\{0,1\}$. The cover $\widetilde{P}$ has the property that the map $N \to P$ lifts uniquely to $\widetilde{P}$ so that $\{0,1\}$ is fixed. Moreover, using that the interior of $\widetilde{R}$ is diffeomorphic to $\bR^2$, there is an embedding $e \colon \widetilde{P} \hookrightarrow N$ fixing $\{0,1\}$ such that the composition $N \to \widetilde{P} \to N$ is isotopic to the identity relative to $\{0,1\}$. Viewing $\widetilde{P}$ as being glued together by three parts---$N$, the two half-strips resulting from the cut $1$-handle, and the two copies of $\tilde{R}$ attached to these two half-strips---this embedding $e \colon \widetilde{P} \hookrightarrow N$ is given by the identity on $N\subset \widetilde{P}$ apart from a neighbourhood of the two arcs in the boundary to which the half-strips are attached, and by pushing the half-strips and the copies of $\tilde{R}$ attached to them into this neighbourhood.
	
	\begin{figure}
		\centering
		\begin{tikzpicture}
			\draw [dotted] (0,0) circle (2cm);
			\draw [domain=160:20] plot ({2*cos(\x)}, {2*sin(\x)});
			\draw [domain=-160:-20] plot ({2*cos(\x)}, {2*sin(\x)});
			\draw ({2*cos(20)},{2*sin(20)}) -- ({2*cos(20)+1},{2*sin(20)});
			\draw ({2*cos(160)},{2*sin(160)}) -- ({2*cos(160)-1},{2*sin(160)});
			\draw ({2*cos(-20)},{2*sin(-20)}) -- ({2*cos(-20)+1},{2*sin(-20)});
			\draw ({2*cos(-160)},{2*sin(-160)}) -- ({2*cos(-160)-1},{2*sin(-160)});
			
			\draw [ultra thick,Mahogany] (0,-2) -- (0,2);
			\node [right] at (0,0) {$\tilde{\alpha}$};
			\draw [very thick,Mahogany,dashed] ({2*cos(20)+1},{2*sin(20)}) -- ({2*cos(-20)+1},{2*sin(-20)});
			\draw [very thick,Mahogany,dashed] ({2*cos(160)-1},{2*sin(160)}) -- ({2*cos(-160)-1},{2*sin(-160)});
			
			\draw[yshift=-.3cm] (-1.5,-.5) to[out=0,in=-90] (-1.4,0.8) to[out=90,in=180] (-1,1.2) to[out=0,in=90] (-.6,0.8) to[out=-90,in=180] (-.5,-.5);	
			\draw[yshift=-.3cm] (-1.2,0.8) to[out=-90,in=-90] (-0.8,0.8);
			\draw[yshift=-.3cm] (-1.15,0.75) to[out=90,in=90] (-0.85,0.75);
			
			\draw[xshift=2cm,yshift=.1cm] (-1.5,-.5) to[out=0,in=-90] (-1.4,0.8) to[out=90,in=180] (-1,1.2) to[out=0,in=90] (-.6,0.8) to[out=-90,in=180] (-.5,-.5);	
			\draw[xshift=2cm,yshift=.1cm] (-1.2,0.8) to[out=-90,in=-90] (-0.8,0.8);
			\draw[xshift=2cm,yshift=.1cm] (-1.15,0.75) to[out=90,in=90] (-0.85,0.75);

			\node at (1,-1) {$R$};
		\end{tikzpicture}
		\caption{The surface $R$.}
		\label{fig:tildegamma}
	\end{figure}
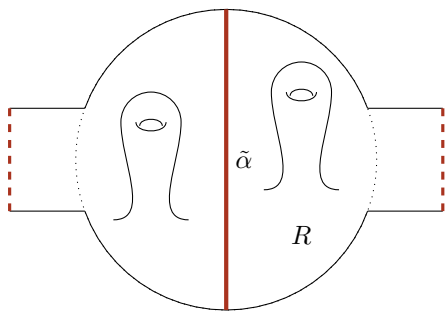
	
	The right square is induced by post-composition with $e$, so homotopy commutes in view of \cref{sec:formal-properties} \ref{postcomposition}. The homotopy commutative left square is obtained by invoking the lifting lemma \cref{lem:lift-arcs} for the covering map $\widetilde{P} \to P$. The top composition in \eqref{eqn:retraction-diagram} is homotopic to the identity by construction, but it remains to justify this for the bottom composition. Justifying this requires the details of the proof of \cref{lem:lift}, in particular the presheaf $\Emb^\pi_\partial(-,\widetilde{P})$ defined there. Viewing $N$ as a submanifold of $\widetilde{P}$ as explained above, the projection $\pi\colon \widetilde{P}\ra N$ is isotopic to the identity when restricted to $N$, so we have a dashed inclusion map of presheaves on $\cat{Disc}_{\partial I}$ that makes the triangle in the following diagram commute up homotopy
	\[\begin{tikzcd} \Emb_{\partial}(-,N) \rar[hook] \arrow[dashed,"\subset",swap]{rd}& \Emb_{\partial}(-,P) &&\\
		& \Emb^\pi_{\partial}(-,\widetilde{P}) \uar[swap]{\pi \circ -} \rar[hook] &\Emb_{\partial}(-,\widetilde{P})\rar{e \circ (-)}&\Emb_{\partial}(-,N)
	\end{tikzcd}.\]
	Moreover, since $N \subset \widetilde{P} \to N$ is isotopic to the identity, the composition $\Emb_{\partial}(-,N)\ra\Emb_{\partial}(-,N)$ along the bottom is homotopic to the identity. Applying $T_\infty$, evaluating at $I$, and restricting to path-components, we obtain a homotopy commutative diagram
		\[\begin{tikzcd} T_\infty\Emb_{\partial}(I,N)_\alpha \rar\arrow{rd}& T_\infty\Emb_{\partial}(I,P)_{\tilde{\alpha}} &&\\
		& T_\infty\Emb^\pi_{\partial}(I,\widetilde{P})_{\tilde{\alpha}} \uar[swap]{\simeq} \rar &T_\infty\Emb_{\partial}(I,\widetilde{P})_{\tilde{\alpha}}\rar{e \circ (-)}&T_\infty\Emb_{\partial}(I,N)_\alpha
	\end{tikzcd}\]
	whose composition along the bottom $T_\infty\Emb_{\partial}(-,N)_\alpha\ra T_\infty\Emb_{\partial}(-,N)_\alpha$ is homotopic to the identity. The composition along the top involving a wrong-way weak equivalence agrees by construction with the bottom composition of \eqref{eqn:retraction-diagram}, so it is homotopic to the identity as claimed (recall \cref{rem:weak-map}).

\subsection*{Step \ref{enum:step-circle}}\textit{Convergence for $(M,N)$ if $M$ is a circle, cylinder, or Möbius strip, and $\dim(N)=2$.}

\medskip

\noindent By \cref{lem:thick-arc-converse}, it suffices to prove the claim for the cylinder and the Möbius strip. We will do so for the Möbius strip $M=\mr{Mo}$; the argument for the cylinder is analogous. We pick a disc $D^2 \subset \mr{int}(\mr{Mo})$. By \cref{lem:remove-discs} it suffices to prove that 
\[\Emb_{\partial_0}(\mr{Mo} \backslash \mr{int}(D^2),N \backslash \mr{int}(e(D^2))) \lra T_\infty \Emb_{\partial_0}(\mr{Mo} \backslash \mr{int}(D^2),N \backslash \mr{int}(e(D^2)))\]
is a weak equivalence for all embeddings $e \colon D^2 \hookrightarrow \mr{int}(\Sigma)$. To this end, we pick a subtriad $I \times [0,1] \subset \mr{Mo}\backslash \mr{int}(D^2)$ as in \cref{fig:moebius} and attempt to show that the vertical restriction maps in the diagram
\[\begin{tikzcd} \Emb_{\partial_0}\big(\mr{Mo} \backslash \mr{int}(D^2),N \backslash \mr{int}(e(D^2))\big) \rar \dar & T_\infty \Emb_{\partial_0}\big((\mr{Mo} \backslash \mr{int}(D^2),N \backslash \mr{int}(e(D^2))\big) \dar \\
	\Emb_{\partial_0}\big(I \times [0,1],N \backslash \mr{int}(e(D^2))\big) \rar & T_\infty \Emb_{\partial_0}\big(I \times [0,1],N \backslash \mr{int}(e(D^2))\big)\end{tikzcd}\]
are weak equivalences. Isotopy extension exhibits the homotopy fibre of the left vertical map up to smoothing corner and isotopy equivalence as $\Emb_\partial(I \times [0,1) \sqcup I \times [0,1),N \backslash \mr{int}(e(D^2)))$ which is contractible by the contractibility of spaces of collars. To see that the right vertical map is an equivalence, one combines this observation with descent with respect to a Weiss $\infty$-cover of open discs and collars on $\partial D^2$ similarly to the proof of \cref{lem:forget-collar}. As the bottom horizontal map is a weak equivalence by Step \ref{enum:step-interval}, the top horizontal map is a weak equivalence as well.

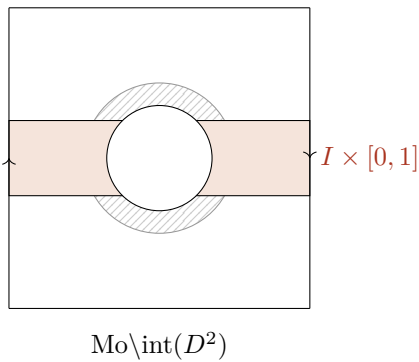
\begin{figure}
	\begin{tikzpicture}
		[
		decoration={
			markings,
			mark=at position 0.5 with {\arrow{>}}}
		] 
		\draw[black!40!white,pattern=north east lines,pattern color=black!20!white] (0,0) circle (1cm);
		\draw[fill=Mahogany!10!white] (-2,-.5) -- (2,-.5) -- (2,.5) -- (-2,.5) -- cycle;
		\draw (-2,2) -- (2,2);
		\draw[postaction=decorate] (2,2) -- (2,-2);
		\draw (2,-2) -- (-2,-2);
		\draw[postaction=decorate] (-2,-2) -- (-2,2);
		\draw [fill=white](0,0) circle (.7cm);
		\node at (0,-2.5) {$\mr{Mo}\backslash \mr{int}(D^2)$};
		\node [Mahogany] at (2.8,0) {$I \times [0,1]$};
	\end{tikzpicture}
	\caption{The complement of an open disc in Möbius strip. The orange copy of $I \times [0,1]$ differs up to isotopy equivalence from $\mr{Mo}\backslash \mr{int}(D^2)$ only in hatched region which is diffeomorphic to $I \times [0,1) \sqcup I \times [0,1)$.}
	\label{fig:moebius}
\end{figure}

\subsection*{Step \ref{enum:step-line-over-1dim}}\textit{Convergence for $(M,N)$ if
	\[M = \left(T_1 \times I \times [0,1]\right) \sqcup \left(T_2 \times S^1 \times [0,1]\right) \sqcup \left(T_3 \times \mr{Mo}\right)\]
for (possibly empty) finite sets $T_i$ and $\dim(N)=2$.}

\medskip

\noindent The proof is by induction over $t = |T_1|+|T_2|+|T_3|$. The initial case $t=1$ is provided by Steps \ref{enum:step-interval} and \ref{enum:step-circle}. For the induction step, we pick a component of $M$, say of the form $I \times [0,1]$; the other cases are analogous. We consider $M'\coloneqq M\backslash  I \times [0,1]$. An application of isotopy extension (see \cref{sec:formal-properties} \ref{isotopy-extension}) to $P = I \times [1/4,3/4] \subset I \times [0,1]$, justified by \cref{lem:thick-arc-discs} and Step \ref{enum:step-interval}, gives fibre sequences
\[\begin{tikzcd} \Emb_{\partial_0}(M',N \backslash \mr{int}(e(P))) \dar \rar & \Emb_{\partial_0}(M,N) \rar \dar & \Emb_{\partial_0}(P,N) \dar{\simeq} \\
T_\infty \Emb_{\partial_0}(M',N \backslash \mr{int}(e(P))) \rar & T_\infty \Emb_{\partial_0}(M,N) \rar & T_\infty \Emb_{\partial_0}(P,N).\end{tikzcd}\]
Here we used \cref{lem:forget-collar} and isotopy invariance to replace $M' \sqcup (I \times [0,1] \backslash \mr{int}(P))$ in the domain with $M'$. The left vertical map is a weak equivalence by the induction hypothesis, so the middle vertical map is a weak equivalence one too.

\subsection*{Step \ref{enum:step-1dim}}\textit{Convergence for $(M,N)$ if $\dim(M)=1$ and $\dim(N)=2$.}

\medskip

\noindent Step \ref{enum:step-line-over-1dim} together with \cref{lem:thick-arc-converse} gives result for those triads of the form $M' = (T_1 \times I) \sqcup (T_2 \times S^1)$ for finite sets $T_i$ and $\partial_0 M' = T_1 \times \{0,1\}$. The general case, which has
\[M = (T_1 \times I) \sqcup (T_2 \times S^1)  \sqcup (T_3 \times [0,1]) \sqcup (T_4 \times [0,1])\]
for finite sets $T_i$ and $\partial_0 M = (T_1 \times \{0,1\}) \sqcup (T_3 \times \{0\})$ follows from this by Lemmas \ref{lem:forget-collar} and \ref{lem:thick-arc-discs} together with isotopy invariance (see \cref{sec:formal-properties} \ref{naturality}).

\subsection*{Step \ref{enum:step-disc}}
\textit{Convergence for $(M,N)$  if $M=D^2$ with $\partial_0M=\partial M$, and $\dim(N)=2$.}

\medskip

\noindent By \cref{cor:add-components} we may assume that $N$ is connected. 

We first prove the case where the target $N$ is \emph{not} diffeomorphic to $D^2$. In this case $\Emb_\partial(D^2,N) = \varnothing$, so we need to show $T_\infty\Emb_\partial(D^2,N)=\varnothing$. If this were to fail, then the target of the map $T_\infty \Emb_\partial(D^2,N) \ra \Map_\partial(D^2,N)$ from \cref{sec:formal-properties} \ref{mapping-space} must be nonempty, so $N$ would be a connected surface with a boundary component whose inclusion is null-homotopic. We claim this is impossible unless $N \cong D^2$. Firstly, if $N = N_1 \natural \cdots \natural N_1$, then $\pi_1(N)$ splits as a free product $\pi_1(N_1) \ast \cdots \ast \pi_1(N_n)$ and we may choose this decomposition so that the homotopy class of the boundary inclusion represents the free product of the homotopy classes of boundary inclusions of those components at which we perform the boundary connected sums. By the classification of connected compact surfaces, it then suffices to observe that all boundary inclusions are non-trivial in the fundamental group of the surfaces $\Sigma_{0,2}$, $\Sigma_{1,1}$, and $\mr{Mo}$. For $\Sigma_{0,2}$, each inclusion represents a generator of $\pi_1(\Sigma_{0,2}) \cong \bZ$, for $\Sigma_{1,1}$ the boundary inclusion represents $xyx^{-1}y^{-1} \in \pi_1(\Sigma_{1,1}) \cong \langle x,y\rangle$, and for the Möbius strip it represents twice a generator in $\pi_1(\mr{Mo}) \cong \bZ$.

It remains to show that $\Emb_\partial(D^2,D^2) \ra T_\infty \Emb_\partial(D^2,D^2)$
a weak equivalence for which we follow the proof of what is sometimes called the \emph{Cerf Lemma} \cite[Proposition 5]{CerfAppl}. We consider the triad $H = D^2 \cap ([-1/2,\infty) \times \bR)$ with $\partial_0 H = H \cap \partial D^2$ and $\partial_1 H = H \cap (\{-1/2\} \times \bR)$ containing the strip $J = H \cap ([-1/4,1/4] \times \bR)$ with  $\partial_0J=J\cap\partial D^2$, see Figure \ref{fig:disc:triads}. Writing $H_0 = H \setminus ((-1/4,1/4) \times \bR)\cap H$ and $D^2_0 = D^2 \setminus ((-1/4,1/4) \times \bR)\cap D^2$, an application of isotopy extension (see \cref{sec:formal-properties} \ref{isotopy-extension}) justified by Step \ref{enum:step-interval} in the case $M = J \cong I \times [0,1]$ and \cref{lem:thick-arc-discs} gives a map of fibre sequences with connected weakly equivalent bases and homotopy fibres over the standard inclusion $J \hookrightarrow D^2$
\[\begin{tikzcd} \Emb_{\partial_0}(H_0,D^2_0) \rar \dar & \Emb_{\partial_0}(H,D^2) \dar \rar & \Emb_{\partial_0}(J,D^2) \dar{\simeq}  \\ 
	T_\infty \Emb_{\partial_0}(H_0,D^2_0) \rar & T_\infty \Emb_{\partial_0}(H,D^2) \rar & T_\infty \Emb_{\partial_0}(J,D^2) \end{tikzcd}.\]
As $H$ is a closed collar on $\partial_0H$, the middle vertical map is a weak equivalence by isotopy invariance and the convergence on collars (see \cref{sec:formal-properties} \ref{naturality} and \ref{convergence}). By \cref{lem:forget-collar} we may discard the collar $H_0 \cap ((-\infty,0] \times \bR)$ from the source of the left vertical map, and obtain that for $H'_0 = H \cap ([1/4,\infty) \times \bR)$ the map $\Emb_{\partial_0}(H'_0,D^2_0) \ra T_\infty \Emb_{\partial_0}(H'_0,D^2_0)$ is a weak equivalence. Invoking \cref{cor:add-components} to neglect $D_0^2\backslash H_0$ from the target and identifying $H'_0$ with a disc upon smoothing corners, we conclude that $\Emb_\partial(D^2,D^2) \ra T_\infty \Emb_\partial(D^2,D^2)$ is a weak equivalence.

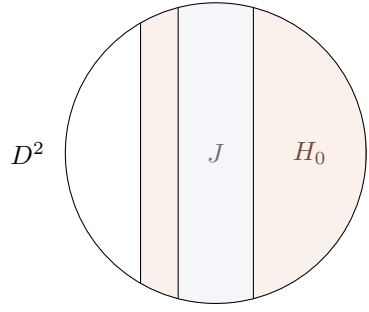
\begin{figure}
	\begin{tikzpicture}[scale=1]
		\node at (-2.5,0) {$D^2$};
		\begin{scope}
			\clip (0,0) circle (2cm);
			\draw[fill=Mahogany!5!white] (-1,3) -- (-1,-3) -- (2,-3) -- (2,3) -- cycle;
			\draw[fill=Periwinkle!5!white] (-.5,3) -- (-.5,-3) -- (.5,-3) -- (.5,3) -- cycle;
		\end{scope}
		\draw (0,0) circle (2cm);
		\node[Periwinkle!50!black] at (0,0) {$J$};
		\node[Mahogany!50!black] at (1.25,0) {$H_0$};
	\end{tikzpicture}
	\caption{The triads $J,H_0 \subset D^2$. Here $H$ the union of $J$ and $H_0$, and $H'_0 \subset H_0$ is the component to the right of $J$.}
	\label{fig:disc:triads}
\end{figure}

\subsection*{Step \ref{enum:step-genus0}}
\textit{Convergence for $(M,N)$ if $M$ is an orientable surface of genus $0$ with $n \geq 1$ boundary components and $\partial_0 M = \partial M$, and $\dim(N)=2$.}

\medskip

\noindent Note that by gluing $(n-1)$ discs to $M$ we obtain a disc $D^2$. We also glue $n-1$ discs to the corresponding boundary components of $N$ to obtain a triad $N'$ with a canonical embedding $e \colon \ul{n{-}1} \times D^2 \hookrightarrow N'$. Then isotopy extension and the convergence on discs (see \cref{sec:formal-properties} \ref{isotopy-extension} and \ref{convergence}) yields fibre sequences
\[\begin{tikzcd} \Emb_\partial(M,N) \rar \dar & \Emb_\partial(D^2,N') \rar \dar & \Emb(\ul{n{-}1} \times D^2,N') \dar{\simeq} \\
T_\infty \Emb_\partial(M,N) \rar & T_\infty \Emb(D^2,N') \rar & T_\infty \Emb(\ul{n{-}1} \times D^2,N').\end{tikzcd}\]
The middle vertical map a weak equivalence by Step \ref{enum:step-disc}, so the left map is one as well.

\subsection*{Step \ref{enum:step-connected-full-bdy}}\textit{Convergence for $(M,N)$ if $M$ is connected, $\partial_0M=\partial M$, and $\dim(M)=\dim(N)=2$.}

\begin{figure}
\centering
\begin{minipage}[b]{0.52\textwidth}
  \centering
  \begin{tikzpicture}[scale=0.8]
		\draw (0,0) circle (3cm);
		
		\begin{scope}[yshift=-.4cm,xshift=1.2cm,scale=1.3]
		\draw [fill=Mahogany!20!white] (-1.1,0.8) -- (-1.1,1.19) to[out=10,in=-190] (-0.9,1.19) -- (-0.9,0.8) -- cycle;
		\draw (-1.5,-.5) to[out=0,in=-90] (-1.4,0.8) to[out=90,in=180] (-1,1.2) to[out=0,in=90] (-.6,0.8) to[out=-90,in=180] (-.5,-.5);	
		\draw (-1.2,0.8) to[out=-90,in=-90] (-0.8,0.8);
		\draw [fill=white] (-1.15,0.75) to[out=90,in=90] (-0.85,0.75);
		\end{scope}
	
		\begin{scope}[yshift=-.4cm,xshift=-.3cm,scale=1.3]
		\draw [fill=Mahogany!20!white] (-1.1,0.8) -- (-1.1,1.19) to[out=10,in=-190] (-0.9,1.19) -- (-0.9,0.8) -- cycle;
		\draw (-1.5,-.5) to[out=0,in=-90] (-1.4,0.8) to[out=90,in=180] (-1,1.2) to[out=0,in=90] (-.6,0.8) to[out=-90,in=180] (-.5,-.5);	
		\draw (-1.2,0.8) to[out=-90,in=-90] (-0.8,0.8);
		\draw [fill=white] (-1.15,0.75) to[out=90,in=90] (-0.85,0.75);
		\end{scope}
		
		\draw (1.3,1.5) circle (.45cm);
		\draw (1.3,0) circle (.45cm);
		\draw (1.3,-1.5) circle (.45cm);
	\end{tikzpicture}
	\caption{$M=\Sigma_{0,4}\,\natural\,(\Sigma_{1,1})^{\natural \ul{2}}$ with subtriad $P = \ul{2} \times S^1 \times [0,1]\subset M$ whose complement has genus $0$ and $8$ boundary components.}
	\label{fig:arcs}
\end{minipage}%
\begin{minipage}[b]{0.52\textwidth}
  \centering
\begin{tikzpicture}[scale=.8]
		\draw (0,0) circle (3cm);
		
		\begin{scope}[yshift=-.4cm,xshift=-.2cm,scale=1.3]
			\draw (-1.5,-.5) to[out=0,in=-90] (-1.4,0.8) to[out=90,in=180] (-1,1.2) to[out=0,in=90] (-.6,0.8) to[out=-90,in=180] (-.5,-.5);	
			\draw (-1.2,0.8) to[out=-90,in=-90] (-0.8,0.8);
			\draw (-1.15,0.75) to[out=90,in=90] (-0.85,0.75);
		\end{scope}

		\draw [fill=Mahogany!20!white] (1.3,1) circle (.9cm);
		\draw [fill=white] (1.3,1) circle (.75cm);
		\draw [very thick,Mahogany,dashed] (1.3,1) circle (.6cm);
		
		\begin{scope}[yshift=-.2cm,xshift=.1cm]
			\draw [fill=Mahogany!20!white] (.9,-1.4) circle (.75cm);
			\draw [fill=white] (.9,-1.4) circle (.6cm);
			\draw [fill=Mahogany!20!white] (1.7,-.6) circle (.75cm);
			\draw [fill=white] (1.7,-.6) circle (.6cm);
			\fill [white] (1.3,-1) circle (.6cm);
			\draw [domain=90:0,very thick,Mahogany,dashed] plot ({.6*cos(\x)+1.3}, {.6*sin(\x)-1});
			\draw [domain=270:180,very thick,Mahogany,dashed] plot ({.6*cos(\x)+1.3}, {.6*sin(\x)-1});
			\draw [domain=180:90] plot ({.6*cos(\x)+1.3}, {.6*sin(\x)-1});
			\draw [domain=360:270] plot ({.6*cos(\x)+1.3}, {.6*sin(\x)-1});
		\end{scope}
	\end{tikzpicture}
	\caption{$M = \Sigma_{0,3}\,\natural\,\Sigma_{1,1}$ with $\partial_1 M$ dashed, with subtriad $P = \ul{2} \times I \times [0,1] \sqcup S^1 \times [0,1] \subset M$. The component of $M \backslash \mr{int}(P)$ containing $\Sigma_{1,1}$ is $M'$.}
	\label{fig:boundary}
\end{minipage}
\end{figure}

\medskip
\noindent As a result of \cref{lem:remove-discs}, we may assume that $\partial M \neq \varnothing$, so $M$ is a boundary connected sum 
\[\Sigma_{0,n} \,\natural\, (\Sigma_{1,1})^{\natural T_1} \,\natural\, (\bR P^2)^{\natural T_2}\]
for $n\ge1$ and possibly empty finite sets $T_1$ and $T_2$. Thus we may find an embedding
\[P = (T_1 \times S^1 \times [0,1]) \sqcup (T_2 \times \mr{Mo}) \lra M\]
such that $M \backslash \mr{int}(P) \cong \Sigma_{0,n'}$ with $n' = n+2|T_1|+|T_2|$; see \cref{fig:arcs} for an example. For any embedding $e \colon M \hookrightarrow N$ extending the boundary condition, an application of isotopy extension (see \cref{sec:formal-properties} \ref{isotopy-extension}), justified by Step \ref{enum:step-line-over-1dim} and \cref{lem:thick-arc-discs}, gives a map of fibre sequences
\[\begin{tikzcd} \Emb_\partial(\Sigma_{0,n'},N{\backslash} e(\mr{int}(P))) \rar \dar & \Emb_\partial(M,N) \rar \dar & \Emb(P,N) \dar{\simeq} \\
	T_\infty \Emb_\partial(\Sigma_{0,n'},N{\backslash} e(\mr{int}(P))) \rar & T_\infty \Emb(M,N) \rar & T_\infty \Emb(P,N)\end{tikzcd}\]
whose left vertical map a weak equivalence by Step \ref{enum:step-connected-full-bdy}. Varying the embedding $e \colon M \hookrightarrow N$, we conclude that the middle vertical map is also a weak equivalence.

\subsection*{Step \ref{enum:step-connected}}\textit{Convergence for $(M,N)$ if $M$ is connected, $\partial_0M\neq \partial M$, and $\dim(M)=\dim(N)=2$.}

\medskip
\noindent Choose a triad embedding $P = (T_1 \times I \times [0,1]) \sqcup (T_2 \times S^1 \times [0,1]) \hookrightarrow M$
such that $M \backslash \mr{int}(P)$ is the disjoint union of a component $M'$ with $\partial M' = M \cap \partial_0(M \backslash \mr{int}(P))$ and collars on components of $\partial_0(M \backslash \mr{int}(P))$; see \cref{fig:boundary} for an example. By Step \ref{enum:step-line-over-1dim} and \cref{lem:thick-arc-discs}, we may apply isotopy extension as in Step \ref{enum:step-line-over-1dim} to the restriction map $\Emb_{\partial_0}(M,N) \ra \Emb_{\partial_0}(P,N)$ and its $T_\infty$-version. From Step \ref{enum:step-connected-full-bdy} and \cref{lem:forget-collar} we see that the map between fibres is a weak equivalence, from which we conclude the claim.

\subsection*{Step \ref{enum:step-general}}\textit{Convergence for $(M,N)$ if $\dim(M)=\dim(N)=2$}

\medskip
\noindent This is a induction on the number $n$ of components of $M$. The initial case $n=1$ is the previous one, and for the induction step we write $M = M' \sqcup M''$ with $M'$ connected. The induction hypothesis applied to $M'$ together with \cref{lem:thick-arc-discs} ensures that we may apply isotopy extension (see \cref{sec:formal-properties} \ref{isotopy-extension}) to the restriction $\Emb_{\partial_0}(M,N)\ra \Emb_{\partial_0}(M',N)$ and its $T_\infty$-version from which the claim follows by noting that the map on fibres is a weak equivalence by applyinh the induction hypothesis to $M''$.

\subsection*{Step \ref{enum:step-both-1dim}}\textit{Convergence for $(M,N)$ if $\dim(M)=\dim(N)=1$}
This can be proved similarly to Step \ref{enum:surfaces} but is easier. We outline the argument.

First one proves the case $M=D^1$ with $\partial_0(D^1)=\{-1,1\}$ by a strategy analogously to Step \ref{enum:step-disc}: one first uses \cref{cor:add-components} to reduce to $N=D^1$ as in the case for surfaces. Then one takes $H = [-1/2,1]$, $J = [-1/4,1/4]$, and $D^1_0 = D^1 \backslash \mr{int}(J)$ and develops a map of fibre sequences
\[\begin{tikzcd} \Emb_{\partial_0}(H_0,D^1_0) \rar \dar & \Emb_{\partial_0}(H,D^1) \dar{\simeq} \rar & \Emb_{\partial_0}(J,D^1) \dar{\simeq}  \\ 
	T_\infty \Emb_{\partial_0}(H_0,D^1_0) \rar & T_\infty \Emb_{\partial_0}(H,D^1) \rar & T_\infty \Emb_{\partial_0}(J,D^1). \end{tikzcd}\]
similar to Step \ref{enum:step-disc}. Using \cref{lem:forget-collar} and \cref{cor:add-components}, the map on fibres agrees with 
$\Emb_\partial(D^1,D^1) \ra T_\infty \Emb_\partial(D^1,D^1)$,
so it is a weak equivalence.

Next one shows the case of a general connected triad $M$: the case $M=S^1$ follows directly by an application of isotopy extension (see \cref{sec:formal-properties} \ref{isotopy-extension}) together with the case $M=D^1$ above, and the cases $M=[0,1]$ with $\partial_0(M)=\{0\}$ or $\partial_0(M)=\varnothing$ hold by \cref{sec:formal-properties} \ref{convergence}. 

Finally, the case of a possibly disconnected triad $M$ can be settled as in Step \ref{enum:step-connected}.

\subsection{Automorphisms of the $E_1$- and $E_2$-operad}\label{sec:aut-little-discs}The above arguments do not rely on the fact that $\Diff_\partial(D^d)=\Emb_\partial(D^d,D^d)$ is contractible for $d\le 2$ (this is folklore for $d=1$ and due to Smale for $d=2$ \cite{Smale}). Using this fact, we may conclude from \cref{thm:main-general-triads} that $T_\infty\Emb_\partial(D^d,D^d)$ is contractible for $d\le2$. Combining Theorems 1.2, 1.4, and 6.4 of \cite{BoavidaWeissConf}, we have \[T_\infty \Emb_\partial(D^d,D^d)\simeq \Omega^{d+1}\mr{Aut}^h(E_d)/\mr{O}(d)\] where $\mr{Aut}^h(E_d)/\mr{O}(d)$ is the homotopy fibre of the map $B\mr{O}(d)\ra B\mr{Aut}^h(E_d)$
resulting from the standard action of $\mr{O}(d)$ on the little discs operad by derived operad automorphisms, so we deduce:

\begin{corollary}\label{cor:haute2} $\Omega^{d+1}\mr{Aut}^h(E_d)/\mr{O}(d)\simeq *$ for $d\le 2$.
\end{corollary}

\begin{remark}Horel \cite[Theorem 8.5]{HorelAut} proved that $\mr{Aut}^h(E_2)/\mr{O}(2)\simeq *$ with different methods. His proof crucially uses the spaces of $k$-arity operations in the operad $E_2$ are $K(\pi,1)$ for all $k$. This fact can also be used to give an alternative proof of $\Omega^2\mr{Aut}^h(E_2)\simeq \ast$ (and thus of \cref{cor:haute2}): the derived mapping space $\mathrm{Map}^h(O,P)$ between operads $O$ and $P$ can be computed as a homotopy limit of a diagram whose values are products of spaces of operations in $O$ and $P$; this follows e.g.\,by using the alternative model of operads in terms of dendroidal Segal spaces. Applied to $O=P=E_2$, one sees that $\mathrm{Map}^h(E_2,E_2)$ is a homotopy limit of $K(\pi,1)$, so it is contractible after looping twice.
\end{remark}

\section{Embedding calculus and the Johnson filtration}\label{sec:johnson} 
This section serves to introduce the filtration \eqref{equ:filtration-embcalc} of the mapping class group $\pi_0\,\Diff_\partial(\Sigma_{g,1})$, and to prove in \cref{thm:johnson-filtration} that it is contained in the Johnson filtration.

\subsection{The cardinality filtration} Returning to the general setting of manifold calculus of \cref{sec:presheaf-model} with a fixed $(d-1)$-manifold $K$, possibly with boundary, we consider the filtration 
\begin{equation}\label{equ:disc-filtration}
\cat{Disc}_{{\partial_0},\le0}\subset\cat{Disc}_{{\partial_0},\le1}\subset\cdots\subset \cat{Disc}_{{\partial_0},\le\infty}\coloneqq\cat{Disc}_{{\partial_0}}
\end{equation}
of the topologically enriched category $\cat{Disc}_{{\partial_0}}$ by its full subcategories $\cat{Disc}_{{\partial_0},\le k}$ on triads that are diffeomorphic to $K\times [0,1)\sqcup T\times \mathbb{R}^d$ for finite sets $T$ of bounded cardinality $\le k$. Localising the categories $\cat{PSh}(\cat{Disc}_{\partial_0,\le k})$ at the objectwise weak equivalences as we did for $k=\infty$ in \cref{sec:presheaf-model}, given a presheaf $F\in\cat{PSh}(\cat{Disc}_{\partial_0})$ we obtain presheaves on $\cat{Man}_{\partial_0}$ by
\[T_kF(M)\coloneqq  \Map_{\cat{PSh}(\cat{Disc}_{\partial_0,\le k})^{\loc}}(\Emb_{\partial_0}(-,M),F )\]
which are related by maps of presheaves
\begin{equation}\label{equ:cardinality-filtration-mfd-calc}
T_\infty F(M)\lra\cdots\lra T_2 F(M)\lra T_1F(M)
\end{equation}
induced by restriction along the inclusions \eqref{equ:disc-filtration}. If $F$ is the restriction of a presheaf on $\cat{Man}_{\partial_0}$, we can precompose this tower with the canonical map $F(M)\ra T_\infty F(M)$ from \eqref{equ:manifold-calculus-taylor}.

\subsubsection{Sheaf-theoretic point of view}The tower \eqref{equ:cardinality-filtration-mfd-calc} can also be seen from the point of view of $\cJ_k$-sheaves as described in \cref{sec:formal-properties-mfd} \ref{descent}: by \cite[Theorem 1.2]{BoavidaWeissSheaves} the functor
\[\cat{PSh}(\cat{Man}_{\partial_0})\ni F\longmapsto T_k F\in \cat{PSh}(\cat{Man}_{\partial_0})\]
together with the natural transformation $\mr{id}_{\cat{PSh}(\cat{Man}_{\partial_0})}\Rightarrow T_k$ is a model for the homotopy $\cJ_k$-sheafification. From this point of view the maps \eqref{equ:cardinality-filtration-mfd-calc} are induced by the universal property of homotopy sheafification, using the fact that any $\cJ_{k+1}$-sheaf is in particular a $\cJ_{k}$-sheaf. 

In particular, in the case of embedding calculus, i.e.\,for presheaves $F(-)=\Emb_{\partial_0}(-,N)$ for triads $M$ and $N$ and a boundary condition $e_{\partial_0}\colon \partial_0M\hookrightarrow \partial_0N$ (see Example~\ref{ex:emb-calc}), this implies that there is a factorisation of the map from the discussion in \cref{sec:formal-properties} \ref{mapping-space} of the form
\begin{equation}\label{equ:t1-factorisation-mapping-space}
T_\infty\Emb_{\partial_0}(M,N)\lra\cdots\lra T_1\Emb_{\partial_0}(M,N)\lra \Bun_{\partial_0}(TM,TN)\lra \Map_{\partial_0}(M,N).
\end{equation}
\subsection{$H\bZ$-embedding calculus} 
Much of the above goes through for presheaves valued in categories other than spaces. We have use for one such generalisation, which we discuss now.

It involves the topologically enriched category $\cat{Sp}$ of spectra and the topologically enriched category  $H\bZ\text{-}\cat{mod}$ of module spectra over the Eilenberg--Mac Lane spectrum $H\bZ$, both modelled for example using symmetric spectra in spaces as in \cite{MMSS}. We denote by $\cat{PSh}^{H\mathbb{Z}}(\cat{Disc}_{\partial_0})$ the category of $H\bZ$-module spectrum-valued enriched presheaves on $\cat{Disc}_{\partial_0}$, and its localisation at the objectwise stable equivalences by $\cat{PSh}^{H\bZ}(\cat{Disc}_{\partial_0})^\mr{loc}$. The composition of the left-adjoints  $\Sigma^\infty_+ \colon \cat{Top} \to \cat{Sp}$ and $- \wedge H\bZ \colon \cat{Sp} \to H\bZ\text{-}\cat{mod}$ induces the vertical arrows in the commutative diagram
\begin{equation}
\label{equ:smash-HZ}
\begin{tikzcd} \cat{PSh}(\cat{Disc}_{\partial_0}) \rar \dar[swap]{(-)_+ \wedge H\bZ} & \cat{PSh}(\cat{Disc}_{\partial_0})^\mr{loc} \dar{(-)_+ \wedge H\bZ} \\
\cat{PSh}^{H\bZ}(\cat{Disc}_{\partial_0}) \rar & \cat{PSh}^{H\bZ}(\cat{Disc}_{\partial_0})^\mr{loc}.
\end{tikzcd}
\end{equation}
For a presheaf $F \in \cat{PSh}(\cat{Disc}_{\partial_0})$ we define presheaves 
\[T_k^{H\bZ} F(M) \coloneqq 	\Map_{\cat{PSh}^{H\bZ}(\cat{Disc}_{\partial_0,\le k})^{\loc}}(\Emb_{\partial_0}(-,M)_+ \wedge H\bZ,F_+ \wedge H\bZ)\] for $1\leq k\leq\infty$, giving rise to an extension of the tower \eqref{equ:cardinality-filtration-mfd-calc} to a map of towers
\begin{equation}\label{equ:homology-mfd-calc}
\begin{tikzcd}[column sep=0.5cm ]
T_\infty F(M)\rar\dar&\cdots\rar& T_2 F(M)\rar\dar& T_1F(M)\dar\\
T^{H\bZ}_\infty F(M)\rar&\cdots\rar& T^{H\bZ}_2 F(M)\rar& T^{H\bZ}_1F(M)
\end{tikzcd}
\end{equation}
whose vertical maps are induced by \eqref{equ:smash-HZ} and whose horizontal maps are induced by restriction along \eqref{equ:disc-filtration}. Note that for $F(-)=\Emb_{\partial_0}(-,M)$, composition induces an $E_1$-structure on $T_kF(M)=T_k\Emb_{\partial_0}(M,M)$ and $T_k^{H\bZ} F(M)=T^{H\bZ}_k\Emb_{\partial_0}(M,M)$ which upgrades \eqref{equ:homology-mfd-calc} to a diagram of $E_1$-spaces.

\begin{remark}In \cite{WeissHomology}, Weiss considers manifold calculus applied to the space-valued presheaf $\Omega^\infty(\Emb_{\partial_0}(-,M)_+ \wedge H\bZ)$. This agrees with the above $H\bZ$-embedding calculus since the adjunctions $\Sigma^\infty_+ \dashv \Omega^\infty$ and $- \wedge H\bZ \dashv U$, with $U \colon H\bZ\cat{mod} \to \cat{Sp}$ the forgetful functor, induce adjunctions on presheaf categories, which in turn induce for $F \in \cat{PSh}(\cat{Man}_{\partial_0})$ and $1\leq k\le \infty$ an identification 
\[\begin{tikzcd}\Map_{\cat{PSh}(\cat{Disc}_{\partial_0M,\le k})^{\loc}}(\Emb_{\partial_0}(-,M),\Omega^\infty (F_+ \wedge H\bZ)) \dar{\simeq} \\[-5pt]
	T^{H\bZ}_k F(M)=\Map_{\cat{PSh}^{H\bZ}(\cat{Disc}_{\partial_0M,\le k})^{\loc}}(\Emb_{\partial_0}(-,M)_+ \wedge H\bZ,F_+ \wedge H\bZ)\end{tikzcd}\]
which is compatible with the restrictions maps. 
\end{remark}

\subsection{An $H\bZ$-embedding calculus filtration of $\pi_0\,\Diff_\partial(\Sigma)$} We fix a compact orientable $\Sigma$ of genus $g$ with a single boundary component. A naive attempt at a filtration of the mapping class group $\pi_0\,\Diff_\partial(\Sigma)$ as promised in the introductory \cref{sec:intro-johnson} would be to consider the kernels of the maps $\pi_0\,\Diff_\partial(\Sigma)=\pi_0\,\Emb_\partial(\Sigma,\Sigma)\ra \pi_0\,T_k\Emb_\partial(\Sigma,\Sigma)$ for varying $k$, but these turn out to be trivial for all $k\ge1$ simply because the composition of the above map with the map to $\pi_0\,\Map_{\partial}(\Sigma,\Sigma)$ from \eqref{equ:t1-factorisation-mapping-space} is injective (see \cref{rem:dnb}). To obtain a more interesting filtration, we perform two modifications.

Firstly, we change the triad structure of $\Sigma$. Instead of $\partial_0\Sigma=\partial \Sigma$ we choose $\partial_0\Sigma\subset \partial \Sigma$ to be an embedded interval. We think of $\partial_0\Sigma$ as ``half the boundary'' and abbreviate $\nicefrac{\partial}{2}\coloneqq \partial_0\Sigma\cong [0,1]$. Note that the inclusion $\Diff_\partial(\Sigma)\subset \Emb_{\nicefrac{\partial}{2}}(\Sigma,\Sigma)$ is a homotopy equivalence since its homotopy fibres are equivalent to $\Diff_\partial(D^2)\simeq \ast$. The maps 
$\pi_0\,\Diff_\partial(\Sigma)=\pi_0\,\Emb_{\nicefrac{\partial}{2}}(\Sigma,\Sigma)\ra \pi_0\,T_k\Emb_{\nicefrac{\partial}{2}}(\Sigma,\Sigma)$ still do not give rise to an interesting filtration, for a similar reason as above since the map $\pi_0\,\Map_\partial(\Sigma,\Sigma)\ra \pi_0\,\Map_{\nicefrac{\partial}{2}}(\Sigma,\Sigma)$ is injective. The filtration becomes more interesting after the second modification: we switch from embedding calculus to embedding calculus in $H\bZ$-modules as described above. More precisely, we consider the filtration
\begin{equation}\label{equ:emb-calc-filtration}\pi_0\,\Diff_\partial(\Sigma)= T\mathcal{J}_{\nicefrac{\partial}{2}}^{H\mathbb{Z}}(0)\supset T\mathcal{J}_{\nicefrac{\partial}{2}}^{H\mathbb{Z}}(1)\supset T\mathcal{J}_{\nicefrac{\partial}{2}}^{H\mathbb{Z}}(2)\supset \cdots\end{equation}
defined by  
\[T\mathcal{J}_{\nicefrac{\partial}{2}}^{H\mathbb{Z}}(k)\coloneqq \ker\big[\pi_0\,\Diff_\partial(\Sigma)\ra \pi_0\,T^{H\mathbb{Z}}_k\Emb_{\nicefrac{\partial}{2}}(\Sigma,\Sigma)\big],\] where we formally set $\pi_0\,T^{H\mathbb{Z}}_k\Emb_{\nicefrac{\partial}{2}}(\Sigma,\Sigma)\coloneqq *$. Denoting by
\begin{equation}\label{equ:johnson-filtration}\pi_0\,\Diff_\partial(\Sigma)=\mathcal{J}(0)\supset \mathcal{J}(1)\supset \mathcal{J}(2)\supset \cdots\end{equation} the usual Johnson filtration 
\[\mathcal{J}(k)\coloneqq \ker\left[\pi_0\,\Diff_\partial(\Sigma)\ra\mathrm{Aut}\left(\frac{\pi_1(\Sigma,*)}{\Gamma_k(\pi_1(\Sigma,*))}\right)\right]\]
where $\Gamma_i(-)$ is the $i$th stage in the lower central series of a group (so $\Gamma_0(G)=G$ and $\Gamma_1(G)$ is the derived subgroup of $G$), the purpose of this section is to relate the filtrations \eqref{equ:emb-calc-filtration} and \eqref{equ:johnson-filtration} as follows.

\begin{theorem}\label{thm:johnson-filtration}For a compact orientable surface $\Sigma$ with a single boundary component, the subgroup \[T\mathcal{J}_{\nicefrac{\partial}{2}}^{H\mathbb{Z}}(k)=\ker\big[\pi_0\,\Diff_\partial(\Sigma)\ra \pi_0\,T^{H\mathbb{Z}}_k\Emb_{\nicefrac{\partial}{2}}(\Sigma,\Sigma)\big]\] 
is contained in the $k$th stage $\mathcal{J}(k)$ of the Johnson filtration for $k\ge0$.
\end{theorem}

\begin{remark}\quad
	\begin{enumerate}[label=(\roman*),leftmargin=*]
		\item The group $\pi_1(\Sigma,*)$ is free, so it is residually nilpotent (that is, $\cap_k \Gamma_k(\pi_1(\Sigma,*))=\{1\}$), which implies that the Johnson filtration is exhaustive, i.e.\ $\cap_k \mathcal{J}(k) = \{\mr{id}\}$. By \cref{thm:johnson-filtration}, the same holds for $\{T\mathcal{J}_{\nicefrac{\partial}{2}}^{H\mathbb{Z}}(k)\}$ so in particular the map $\pi_0 \,\Diff_\partial(\Sigma) \to \pi_0\,T^{H\bZ}_\infty \Emb_{\nicefrac{\partial}{2}}(\Sigma,\Sigma)$ is injective.
		\item If the genus of $\Sigma$ is at least $3$, then the inclusion $\smash{T\mathcal{J}_{\nicefrac{\partial}{2}}^{H\mathbb{Z}}(1)}\subset \mathcal{J}(1)$ is strict. Indeed, an element of the mapping class group lies in $T\mathcal{J}_{\nicefrac{\partial}{2}}^{H\mathbb{Z}}(1)$ if and only if induced the identity on the homology of frame bundle $\mr{Fr}(T\Sigma)$. By \cite[Theorem 2.2, Corollary 2.7]{Trapp}, this is the case if and only if it lies in the \emph{Chillingworth subgroup} of the Torelli subgroup $\mathcal{J}(1)$ \cite{ChillingworthI,ChillingworthII}.
	\end{enumerate}
\end{remark}

\cref{thm:johnson-filtration} and the final part of the previous remark suggest:

\begin{question}What is the precise relationship between the Johnson filtration $\mathcal{J}(k)$ and the filtration $\mathcal{J}^{H\bZ}_{\nicefrac{\partial}{2}}(k)$ arising from the $H\bZ$-embedding calculus tower?\end{question}

We will deduce \cref{thm:johnson-filtration} from Moriyama's work \cite{Moriyama}. The key step for this deduction is not special to surfaces and applies to a general $d$-dimensional manifold triad $M$, so we will formulate it in this generality. To do so, we fix a  presheaf $F \in \cat{PSh}(\cat{Disc}_{\partial_0,\leq k})$, restrict it to $\cat{Disc}_{\partial_0,\leq k-1}$ and homotopy left Kan extending it back along the inclusion $\iota_k\colon \cat{Disc}_{\partial_0,\leq k-1}\subset  \cat{Disc}_{\partial_0,\leq k}$ to obtain a presheaf $\mr{hLan}_{\iota_k}\, F$ with a natural map 
$\mr{hLan}_{\iota_k}\, F\ra F$. Evaluating it at 
\[\partial_0 \sqcup \bR^d_\ul{k} \coloneqq \partial_0 M \times [0,1) \sqcup \ul{k} \times \bR^d\] where $\ul{k}\coloneqq\{1,\ldots,k\}$ we get a map of $\Sigma_k \wr O(d)$-spaces
$(\mr{hLan}_{\iota_k}\,F)(\partial_0 \sqcup \bR^d_\ul{k}) \ra F(\partial_0 \sqcup \bR^d_\ul{k})$,
and then taking homotopy quotients by the subgroup $O(d)^k \subset \Sigma_k \wr O(d)$ gives a map
\begin{equation}\label{equ:quotient-hLan}(\mr{hLan}_{\iota_k}\, F)(\partial_0 \sqcup \bR^d_\ul{k})\sslash O(d)^k \lra F(\partial_0 \sqcup \bR^d_\ul{k}) \sslash O(d)^k.\end{equation}
In Proposition~\ref{prop:hlan-fm} below, we relate this map for $F(-)=\Emb_{\partial_0}(-,M)$ to a certain ``boundary inclusion'' of the ordered configuration spaces $\Emb(\ul{k},M)$. For this, recall the \emph{Fulton--Mac Pherson compactification} $\FM_k(M)$ of $\Emb(\ul{k},M)$ (e.g.\,from \cite{Sinha}) which comes with a natural inclusion $\Emb(\ul{k},M) \hookrightarrow \FM_k(M)$ that is homotopy equivalence, and a ``macroscopic location'' map $\mu \colon \FM_k(M) \to M^k$ that extends the inclusion $\Emb(\ul{k},M) \hookrightarrow M^k$. We write $\partial_0 \FM_k(M)$ for the preimage $\mu^{-1}(\Delta_k \cup A_k)$ of the union of the subspace $A_k \subset M^k$ where at least one point lies in $\partial_0 M$ and the fat diagonal $\Delta_k\coloneqq \{(m_1,\ldots,m_i)\in M^k|m_i=m_j\text{ for some }i\neq j\}\subset M^k$.


The key step in the proof of \cref{thm:johnson-filtration} is to identify the map \eqref{equ:quotient-hLan} for $F(-)=\Emb_{\partial_0}(-,M)$ with the boundary inclusion $\partial_0 \FM_k(M) \subset \FM_k(M)$ in the following sense:

\begin{proposition}\label{prop:hlan-fm}There is zig-zag of compatible weak equivalences
\[
\begin{tikzcd}
(\mr{hLan}_{\iota_k}\, \Emb_{\partial_0}(-,M))(\partial_0 \sqcup \bR^d_\ul{k})\sslash O(d)^k\dar\rar{\simeq}&\cdots\dar&\arrow["\simeq",l,swap]\partial_0 \FM_k(M)\dar\\
\Emb(\partial_0 \sqcup \bR^d_\ul{k},M) \sslash O(d)^k\rar{\simeq}&\cdots&\arrow["\simeq",l,swap]\FM_k(M),
\end{tikzcd}
\]
which, when varying $M$, defines a zig-zag of weak equivalences in the arrow category of $\cat{Fun}(\Man_{\partial_0},\cat{S})$.
\end{proposition}

Before turning to the proof of \cref{prop:hlan-fm}, we explain how it implies \cref{thm:johnson-filtration}.

\begin{proof}[Proof of \cref{thm:johnson-filtration}] An element $\phi\in\Diff_\partial(\Sigma)$ induces a commutative diagram
	\begin{equation}\label{equ:diagram-bdy-inc}\begin{tikzcd} \partial_0 \FM_k(\Sigma)_+ \wedge H\bZ \rar{\phi_*} \dar & \partial_0 \FM_k(\Sigma)_+ \wedge H\bZ \dar \\
		\FM_k(\Sigma)_+ \wedge H\bZ \rar{\phi_*} & \FM_k(\Sigma)_+ \wedge H\bZ \end{tikzcd}.\end{equation} Abbreviating $E_\Sigma\coloneqq\Emb_{\partial_0}(-,\Sigma)$, this agrees by \cref{prop:hlan-fm} with the square	
		\[\hspace{-0.5cm}\begin{tikzcd} \big((\mr{hLan}_{\iota_k}\, E_\Sigma)(\partial_0 \sqcup \bR^d_\ul{k})\sslash O(d)^k\big)_+ \wedge H\bZ \rar{\phi_*} \dar & \big((\mr{hLan}_{\iota_k}\, E_\Sigma)(\partial_0 \sqcup \bR^d_\ul{k})\sslash O(d)^k\big)_+ \wedge H\bZ \dar \\
		\big(\Emb(\partial_0 \sqcup \bR^d_\ul{k},\Sigma) \sslash O(d)^k\big)_+ \wedge H\bZ \rar{\phi_*} &\big(\Emb(\partial_0 \sqcup \bR^d_\ul{k},\Sigma) \sslash O(d)^k\big)_+ \wedge H\bZ.\end{tikzcd}
		\] 
		up to a zig-zag of weak equivalences of maps of squares.
As $(-)_+ \wedge H\bZ$ commutes with taking homotopy orbits and left Kan extensions, we conclude that the square \eqref{equ:diagram-bdy-inc} depends up to natural weak equivalences only on the endomorphism $\phi_*\colon \Emb_{\partial_0}(-,\Sigma)_+ \wedge H\bZ\ra\Emb_{\partial_0}(-,\Sigma)_+ \wedge H\bZ$ in $\cat{PSh}^{H\bZ}(\cat{Disc}_{\leq k-1})$ and moreover, as homotopy orbits and homotopy left Kan extensions preserve weak equivalences, only on its image in $\cat{PSh}^{H\bZ}(\cat{Disc}_{\leq k-1})^{\loc}$. Taking vertical cofibres in \eqref{equ:diagram-bdy-inc} and homotopy groups, we conclude that the map
	\begin{equation}\label{equ:map-on-homology-of-fm}\phi_* \colon H_*\big(\FM_k(\Sigma),\partial_0 \FM_k(\Sigma);\bZ\big) \lra H_*\big(\FM_k(\Sigma),\partial_0 \FM_k(\Sigma);\bZ\big)\end{equation}
	depends only on the image of $\phi$ under the map $\pi_0\,\Diff_\partial(\Sigma) \to \pi_0\,T^{H\bZ}_k \Emb_{\partial_0}(\Sigma,\Sigma)$. In particular, if $\phi$ lies in the kernel $T\mathcal{J}^{H\bZ}_{\nicefrac{\partial}{2}}(k)$ of this map, then  \eqref{equ:map-on-homology-of-fm} is the identity. Using excision as in \cite[Section 5.4.1]{KRWalgebraic} one see that the  macroscopic location map $\mu\colon (\FM_k(M),\partial_0 \FM_k(M))\ra (M^k,\Delta_k \cup A_k) $
is a homology isomorphism, so $\phi$ induces the identity on $H_*(M^k,\Delta_k \cup A_k;\bZ)$. But the subgroup of mapping classes with this property is exactly $\mathcal{J}(k)$, by \cite[Theorem A, Proposition 3.3]{Moriyama}.
\end{proof}

\begin{remark}It might be interesting to study the various filtrations of the mapping class group obtained by replacing $H\mathbb{Z}$ in the definition of $T\mathcal{J}_{\nicefrac{\partial}{2}}^{H\mathbb{Z}}(k)$ with $HR$ for any ring $R$, such as $\bQ$ or $\bF_p$.

As long as $R$ has characteristic $0$, the resulting filtration is contained in the Johnson filtration. This follows from the proof for $\bZ$ we gave above, together with the fact from \cite[Proposition 3.3]{Moriyama} that $H_*(\Sigma^k,\Delta_k \cup A_k;\bZ)$ is trivial if $*\neq k$ and free abelian for $*=k$. 
\end{remark}

\subsection{The proof of \cref{prop:hlan-fm}} It will be convenient for us to work with an explicit model for the homotopy left Kan extension as a bar construction, which we recall next.

\subsubsection{The enriched bar construction} Given enriched space-valued functors $F$ and $G$ on a topologically enriched category $\cat{C}$ where $F$ is contravariant and $G$ is covariant, the \emph{bar construction} $B_\bullet(G,\cat{C},F)$ is the semi-simplicial space given by
		\[[p] \longmapsto \bigsqcup_{c_0,\ldots,c_{p}} \Big(G(c_0) \times \prod_{i=1}^p \cat{C}(c_{i-1},c_{i}) \times F(c_p)\Big)\]
where the coproduct is taking over ordered collections $c_0,\ldots,c_p$ of objects in $\cat{C}$ and face maps are induced by the composition in $\cat{C}$ and the functoriality of $F$ and $G$. We denote the geometric realisation of this semi-simplicial space by omitting the $\bullet$-subscript. Since geometric realisations of levelwise weak equivalences of semi-simplicial spaces are weak equivalences, the object $B(G,\cat{C},F)$ is weakly homotopy invariant in  triples $(G,\cat{C},F)$, in the appropriate sense.
		
Given an enriched functor $\iota\colon \cat{C}\ra \cat{D}$ and $d\in\cat{D}$, the space $B\big(\cat{D}(d,\iota(-)),\cat{C},F\big)$ agrees, naturally in $d$, with the homotopy left Kan extension $\mr{hLan}_{\iota}\, F(d)$ (see e.g.\,\cite[Example 9.2.11]{Riehl}; the cofibrancy conditions are not relevant for us as we consider the bar-construction as a \emph{semi}-simplicial space and geometric realisations of \emph{semi}-simplicial spaces preserve weak equivalences). Moreover, if $F$ extends to a functor on $\cat{D}$, then there is a natural augmentation map 
\begin{equation}\label{equ:augmentation}B_\bullet\big(\cat{D}(d,\iota(-)),\cat{C},F\big)\lra F(d)\end{equation}
induced by composition and evaluation, which agrees upon geometric realisations with the canonical map $\mr{hLan}_{\iota}\, F(d)\ra F(d)$ (or rather, it provides a model thereof). 

In particular, using the notation introduced above, the left vertical map in the statement of \cref{prop:hlan-fm} is given by the map induced by \eqref{equ:augmentation} and taking homotopy orbits 
\begin{equation}\label{equ:left-kan-is-bar}
B\big(\Emb_{\partial_0}(\partial_0 \sqcup \bR^d_{\ul{k}},-),\cat{Disc}_{\partial_0,\leq k-1},\Emb_{\partial_0}(-,M)\big)\sslash O(d)^k\xlra{\epsilon} \Emb_{\partial_0}(\partial_0 \sqcup \bR^d_{\ul{k}},M)\sslash O(d)^k.
\end{equation}
To compare \eqref{equ:left-kan-is-bar} to the boundary inclusion $\partial_0\FM_k(M)\subset \FM_k(M)$, we first show the following.
	
	\begin{lemma}\label{lem:fmk-geomrel}The map induced by the augmentation 
	\[ B\big(\partial_0 \FM_k(-),\cat{Disc}_{\partial_0,\leq k-1},\Emb_{\partial_0}(-,M)\big) \lra \partial_0 \FM_k(M)\]
	is a weak equivalence.
\end{lemma}

\begin{proof}[Proof sketch] The strategy is to show that this map is a Serre microfibration and has weakly contractible fibres, which implies the statement by a lemma of Weiss \cite[Lemma 2.2]{WeissClassifying}. This is a standard argument, so we will explain the idea somewhat informally and avoid spelling out lengthy but routine technical details that are similar to e.g.~\cite[Section 4]{KKM}.
	
	To verify that the map is a Serre microfibration the task is to show that in a commutative diagram 
	\[\begin{tikzcd}[ar symbol/.style = {draw=none,"\textstyle#1" description,sloped},
  subset/.style = {ar symbol={\subset}},row sep=1cm,column sep=0.3cm] &D^i\times{\{0\}} \dar \rar & {B(\partial_0 \FM_k(-),\cat{Disc}_{\partial_0,\leq k-1},\Emb_{\partial_0}(-,M))} \arrow[d] \\
		D^i \times \left[0,\varepsilon\right]\arrow[urr,dashed,crossing over, end anchor={south west}]\arrow[r,subset]&D^i \times [0,1] \rar & \partial_0\FM_k(M)\end{tikzcd}\]
		whose solid arrows are given, there exists an $\varepsilon>0$ and dashed lift. To see why this holds, it is helpful to think of the space $B(\partial_0 \FM_k(-),\cat{Disc}_{\partial_0,\leq k-1},\Emb_{\partial_0}(-,M))$ as the subspace of $\partial_0 \FM_k(M) \times B(\ast,\cat{Disc}_{\partial_0,\leq k-1},\Emb_{\partial_0}(-,M))$
	consisting of pairs $(x,[\vec{e},\vec{t}])$ of element $x$ in $\partial_0\FM_k(M)$ and an equivalence class of a collection $\vec{e}$ of $p+1$ levels of nested embedded discs in $M$ with weight $\vec{t} \in \Delta^p$. The pair $[\vec{e},\vec{t}]$ must have the property that the image $\mu(x)$ of $x$ under the macroscopic location map is contained in the interior of the deepest level (see \cref{fig:fmk} for an example) and the equivalence relation is that if a coordinate of $\vec{t} \in \Delta^p = \{(t_0,\ldots,t_p) \in [0,1]^{p+1} \mid t_0+\cdots+t_p = 1\}$ is $0$ then we may forget it and the corresponding level of discs. 
	
	In these terms, the right vertical map in the diagram sends $(x,[\vec{e},\vec{t}])$ to $x$. The map $D^i \to B(\partial_0 \FM_k(-),\cat{Disc}_{\partial_0,\leq k},\Emb_{\partial_0}(-,M))$ provides for each $s \in D^i$ a configuration $x(s) \in \partial_0 \FM_k(M)$ together with nested embedded  discs and weights $[\vec{\epsilon}(s),\vec{t}(s)]$. The map $D^i \times [0,1] \to \partial_0 \FM_k(M)$ defines a homotopy $x_t(s)$ with $t \in [0,1]$ starting at $x(s)$. If $t$ is small enough then this remains within the deepest level of the discs for $x(s,0)$, and by compactness of $D^i$ we find a single $\varepsilon>0$ such that this is the case for all $(s,t)$ with $t \leq \varepsilon$. The dashed lift is then given by sending $(s,t)$ to $(x(s,t),[\vec{e}(s),\vec{t}(s)])$.
	
	To see that the fibre over $x \in \partial_0 \FM_k(M)$ is weakly contractible, i.e.\,any map from $S^i$ to the fibre extends over $D^{i+1}$, we observe that given an equivalence class $[\vec{e},\vec{t}]$ represented by a family of nested embedded discs in $M$ with weights, whose deepest level contains $x$, we find a smaller collection of $\leq (k-1)$ discs around points in the macroscopic image $\mu(x)$ of $x$ and contained in the deepest level. By compactness we can find a single such small collection which works for all images of $s \in S^i$. Adding this collection and transferring all weight to this collection provides an extension to $D^{i+1}$.
\end{proof}

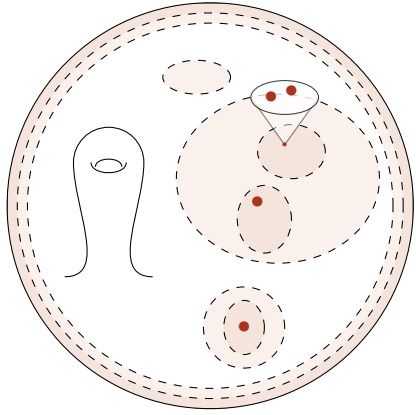
\begin{figure}
	\centering
	\begin{tikzpicture}[scale=.9]
		
	\begin{scope}
		\clip (0,0) circle (3cm);
		\draw [fill=Mahogany!10!white] (-3,-3) rectangle (3,3);
		\draw [fill=Mahogany!5!white,dashed] (0,0) circle (2.85cm);
		\draw [fill=white,dashed] (0,0) circle (2.7cm);
	
	\begin{scope}[yshift=-.4cm,xshift=-.2cm,scale=1.3]
		\draw (-1.5,-.5) to[out=0,in=-90] (-1.4,0.8) to[out=90,in=180] (-1,1.2) to[out=0,in=90] (-.6,0.8) to[out=-90,in=180] (-.5,-.5);	
		\draw (-1.2,0.8) to[out=-90,in=-90] (-0.8,0.8);
		\draw (-1.15,0.75) to[out=90,in=90] (-0.85,0.75);
	\end{scope}
	
	\draw [dashed,fill=Mahogany!5!white] (-0.2,1.9) ellipse (0.5cm and 0.25cm);
	\draw [dashed,fill=Mahogany!5!white] (1,0.4) ellipse (1.5cm and 1.25cm);
	\draw [dashed,fill=Mahogany!5!white] (.5,-1.8) ellipse (0.6cm and .6cm);	
	\draw [dashed,fill=Mahogany!10!white] (1.2,0.8) ellipse (0.5cm and 0.4cm);
	\draw [dashed,fill=Mahogany!10!white] (0.8,-0.2) ellipse (0.4cm and .5cm);
	\draw [dashed,fill=Mahogany!10!white] (.5,-1.8) ellipse (0.3cm and .4cm);		

	\draw [fill=white,opacity=.5] (.6,1.6) -- (1.1,.9) -- (1.6,1.6);
	\draw [fill=white,opacity=.75] (1.1,1.6) ellipse (0.5cm and .25cm);	
	\node at (.9,1.6) [Mahogany] {$\bullet$};
	\node at (1.2,1.7) [Mahogany] {$\bullet$};
	
	\node at (.5,-1.8) [Mahogany] {$\bullet$};
	\node at (1.1,.9) [Mahogany] {$\boldsymbol{\cdot}$};
	\node at (.7,.05) [Mahogany] {$\bullet$};
	
	\end{scope}
	
	\draw (0,0) circle (3cm);
	\end{tikzpicture}
	\caption{An element of $B(\partial_0 \FM_4(-),\cat{Disc}_{\partial_0,\leq 3},\mr{Emb}_{\partial_0}(-,M)$ for $\partial_0 M = \partial M$, consisting of a configuration $x \in \partial_0 \FM_4(M)$ where two points are infinitesimally close, so that its macroscopic image $\mu(x)$ consists of three points, and two levels of discs and collars indicated by the orange and light-orange coloured regions. We suppressed the weights $(t_0,t_1) \in \Delta^1$.}
	\label{fig:fmk}
\end{figure}

\begin{proof}[Proof of \cref{prop:hlan-fm}]
For brevity, we abbreviate
\[\begin{gathered}\cat{D}_{k}\coloneqq \cat{Disk}_{\partial_0,\le k}, \quad E_M\coloneqq\Emb_{\partial_0}(-,M),\quad E^{\partial_0 \sqcup \bR^d_\ul{k}}\coloneqq\Emb_{\partial_0}(\partial_0 \sqcup \bR^d_\ul{k},-),\\
\FM_k = \FM_k(-), \quad \partial_0 \FM_k = \partial_0 \FM_k(-).\end{gathered}\]
We claim that the commutative diagram
\[
\begin{tikzcd}[column sep=0.4cm,row sep=0.3cm]
B\big(E^{\partial_0 \sqcup \bR^d_\ul{k}},\cat{D}_{k-1},E_M\big){\sslash} O(d)^k\arrow[dd]\rar{\circled{1}}&B\big(\FM_k,\cat{D}_{ k-1},E_M\big)\arrow[dd]&B\big(\partial_0\FM_k,\cat{D}_{ k-1},E_M\big)\lar[swap]{\circled{3}}\rar{\circled{4}}\arrow[d]&\partial_0\FM_k(M)\arrow[d,equal]\\
&&\partial_0\FM_k(M)\arrow[d] \arrow[equal]{r}&\partial_0\FM_k(M)\arrow[d]\\
\Emb_{\partial_0}(\partial_0 \sqcup \bR^d_\ul{k},M){\sslash} O(d)^k\rar[swap]{\circled{2}}&\FM_k(M)\arrow[r,equal]&\FM_k(M)\arrow[r,equal]&\FM_k(M)
\end{tikzcd}
\]
provides a zig-zag as claimed. Here all vertical arrows are induced by the augmentation \eqref{equ:augmentation} or the inclusion $\partial_0\FM_k\subset \FM_k$. $\circled{1}$ and $\circled{2}$ are induced by the composition
\begin{equation}\label{equ:comparison-to-fm}\Emb_{\partial_0}(\partial_0 \sqcup \bR^d_\ul{k},-)\ra \Emb(\bR^d_\ul{k},-)\ra \Emb(\ul{k},-)\ra \FM_k(-)\end{equation}
induced by restriction and inclusion, $\circled{3}$ is induced by inclusion, and $\circled{4}$ is another instance of \eqref{equ:augmentation}.  As the diagram is natural in $M$ and the leftmost vertical map agrees with the left vertical map in the statement by the discussion around \eqref{equ:left-kan-is-bar}, it remains show that $\circled{1}$--$\circled{4}$ are weak equivalences. 

The map $\circled{1}$ factors as a composition
\[B\big(E^{\partial_0 \sqcup \bR^d_\ul{k}},\cat{D}_{k-1},E_M\big)\sslash O(d)^k\ra B\big(E^{\partial_0 \sqcup \bR^d_\ul{k}}\sslash O(d)^k,\cat{D}_{k-1},E_M\big)\ra B\big(\FM_k,\cat{D}_{ k-1},E_M\big)\]
whose first map is a weak equivalence since left Kan extensions commute with homotopy orbits. To show that the second map in this composition (and also the map $\circled{2}$) is a weak equivalence, we argue that the composition \eqref{equ:comparison-to-fm} consists of weak equivalences upon applying $(-)\sslash O(d)^k$ to the first two spaces. For the first map this follows by shrinking the collar, for the second map it holds because the derivative $\Emb(\bR^d_\ul{k},N)\ra \ul{k}\times \mathrm{Fr}(N)$ is a weak equivalence for any manifold $N$ where $\mathrm{Fr}(N)$ is the frame bundle, and for the third map it is clear. 

The map $\circled{3}$ is a weak equivalence because $\partial_0 \FM_k(-) \subset \FM_k(-)$ is a weak equivalence when evaluated on objects $U$ of $\cat{D}_{\leq k-1}$. Indeed, if $U$ consists of a collar and $\ell \leq k-1$ discs we have
\[\FM_k(U) \cong \bigsqcup_{n_0+\ldots+n_\ell = k} \FM_{n_0}(\partial_0 M \times [0,1)) \times \FM_{n_1}(\bR^d) \times \cdots \times \FM_{n_\ell}(\bR^d)\]
and $\partial_0 \FM_k(U)$ is the union of such terms where one $\FM_{n_i}$ is replaced by $\partial_0 \FM_{n_i}$. By the pigeonhole principle we have $n_0 \geq 1$ or $n_i \geq 2$ for some $1 \leq i \leq \ell$, so it suffices to observe that in these cases $\partial_0 \FM_{n_0}(\partial M \times [0,1)) \hookrightarrow \FM_{n_0}(M \times [0,1))$ or $\partial_0 \FM_{n_i}(\bR^d) \hookrightarrow \FM_{n_i}(\bR^d)$ are inclusions of deformation retracts, either by modifying configurations such that one has a macroscopic location in $\partial_0 M \times \{0\} \subset \partial_0 M \times [0,1)$ or such that all have macroscopic location at $\{0\} \in \bR^d$. Finally, $\circled{4}$ is a weak equivalence by \cref{lem:fmk-geomrel}.
\end{proof}

\bibliographystyle{amsalpha}
\bibliography{refs2}

\providecommand{\bysame}{\leavevmode\hbox to3em{\hrulefill}\thinspace}
\providecommand{\MR}{\relax\ifhmode\unskip\space\fi MR }
\providecommand{\MRhref}[2]{%
  \href{http://www.ams.org/mathscinet-getitem?mr=#1}{#2}
}
\providecommand{\href}[2]{#2}
\begin{thebibliography}{BdBW18}

\bibitem[BdBW13]{BoavidaWeissSheaves}
P.~Boavida~de Brito and M.~Weiss, \emph{Manifold calculus and homotopy
  sheaves}, Homology Homotopy Appl. \textbf{15} (2013), no.~2, 361--383.
  \MR{3138384}

\bibitem[BdBW18]{BoavidaWeissConf}
\bysame, \emph{Spaces of smooth embeddings and configuration categories}, J.
  Topol. \textbf{11} (2018), no.~1, 65--143. \MR{3784227}

\bibitem[Bol09]{Boldsen}
S.~K. Boldsen, \emph{{Different versions of mapping class groups of surfaces}},
  arXiv:0908.2221.

\bibitem[Cer63]{CerfAppl}
J.~Cerf, \emph{{T}h\'{e}or\`emes de fibration des espaces de plongements.
  {A}pplications}, S\'{e}minaire {H}enri {C}artan, 1962/63, {E}xp. 8,
  Secr\'{e}tariat math\'{e}matique, Paris, 1962/1963. \MR{0160230}

\bibitem[Chi72a]{ChillingworthI}
D.~R.~J. Chillingworth, \emph{Winding numbers on surfaces. {I}}, Math. Ann.
  \textbf{196} (1972), 218--249. \MR{300304}

\bibitem[Chi72b]{ChillingworthII}
\bysame, \emph{Winding numbers on surfaces. {II}}, Math. Ann. \textbf{199}
  (1972), 131--153. \MR{321091}

\bibitem[DK80]{DwyerKan}
W.~G. Dwyer and D.~M. Kan, \emph{Function complexes in homotopical algebra},
  Topology \textbf{19} (1980), no.~4, 427--440. \MR{584566}

\bibitem[Eps66]{Epstein}
D.~B.~A. Epstein, \emph{Curves on {$2$}-manifolds and isotopies}, Acta Math.
  \textbf{115} (1966), 83--107. \MR{214087}

\bibitem[Feu66]{Feustel}
C.~D. Feustel, \emph{Homotopic arcs are isotopic}, Proc. Amer. Math. Soc.
  \textbf{17} (1966), 891--896. \MR{196724}

\bibitem[FM12]{FarbMargalit}
B.~Farb and D.~Margalit, \emph{A primer on mapping class groups}, Princeton
  Mathematical Series, vol.~49, Princeton University Press, Princeton, NJ,
  2012. \MR{2850125}

\bibitem[GK15]{GoodwillieKlein}
T.~G. Goodwillie and J.~R. Klein, \emph{Multiple disjunction for spaces of
  smooth embeddings}, J. Topol. \textbf{8} (2015), no.~3, 651--674.
  \MR{3394312}

\bibitem[Gra73]{Gramain}
A.~Gramain, \emph{Le type d'homotopie du groupe des diff\'{e}omorphismes d'une
  surface compacte}, Ann. Sci. \'{E}cole Norm. Sup. (4) \textbf{6} (1973),
  53--66. \MR{326773}

\bibitem[GW99]{GoodwillieWeiss}
T.~G. Goodwillie and M.~Weiss, \emph{Embeddings from the point of view of
  immersion theory. {II}}, Geom. Topol. \textbf{3} (1999), 103--118.
  \MR{1694808}

\bibitem[Hat14]{HatcherMW}
A.~Hatcher, \emph{A short exposition of the {M}adsen--{W}eiss theorem}, 2014,
  \url{https://pi.math.cornell.edu/~hatcher/Papers/MW.pdf}.

\bibitem[Hor17]{HorelAut}
G.~Horel, \emph{Profinite completion of operads and the
  {G}rothendieck-{T}eichm\"{u}ller group}, Adv. Math. \textbf{321} (2017),
  326--390. \MR{3715714}

\bibitem[KK20]{KnudsenKupers}
B.~Knudsen and A.~Kupers, \emph{{Embedding calculus and smooth structures}},
  arXiv:2006.03109, to appear in {G}eometry \& {T}opology.

\bibitem[KK22]{KrannichKupersConvergence}
M.~Krannich and A.~Kupers, \emph{{The {D}isc-structure space}},
  arXiv:2205.01755.

\bibitem[KKM21]{KKM}
I.~Klang, A.~Kupers, and J.~Miller, \emph{The {M}ay-{M}ilgram filtration and
  {${E}_k$}-cells}, Algebr. Geom. Topol. \textbf{21} (2021), no.~1, 105--136.
  \MR{4224737}

\bibitem[KRW20]{KRWalgebraic}
A.~Kupers and O.~Randal-Williams, \emph{The cohomology of {T}orelli groups is
  algebraic}, Forum Math. Sigma \textbf{8} (2020), Paper No. e64, 52.
  \MR{4190064}

\bibitem[MMSS01]{MMSS}
M.~A. Mandell, J.~P. May, S.~Schwede, and B.~Shipley, \emph{Model categories of
  diagram spectra}, Proc. London Math. Soc. (3) \textbf{82} (2001), no.~2,
  441--512. \MR{1806878}

\bibitem[Mor07]{Moriyama}
T.~Moriyama, \emph{The mapping class group action on the homology of the
  configuration spaces of surfaces}, J. Lond. Math. Soc. (2) \textbf{76}
  (2007), no.~2, 451--466. \MR{2363426}

\bibitem[Rie14]{Riehl}
E.~Riehl, \emph{Categorical homotopy theory}, New Mathematical Monographs,
  vol.~24, Cambridge University Press, Cambridge, 2014. \MR{3221774}

\bibitem[Sin04]{Sinha}
D.~P. Sinha, \emph{Manifold-theoretic compactifications of configuration
  spaces}, Selecta Math. (N.S.) \textbf{10} (2004), no.~3, 391--428.
  \MR{2099074}

\bibitem[Sma59]{Smale}
S.~Smale, \emph{Diffeomorphisms of the {$2$}-sphere}, Proc. Amer. Math. Soc.
  \textbf{10} (1959), 621--626. \MR{112149}

\bibitem[Tra92]{Trapp}
R.~Trapp, \emph{A linear representation of the mapping class group {${\mathcal
  M}$} and the theory of winding numbers}, Topology Appl. \textbf{43} (1992),
  no.~1, 47--64. \MR{1141372}

\bibitem[Tur13]{Turchin}
V.~Turchin, \emph{Context-free manifold calculus and the
  {F}ulton-{M}ac{P}herson operad}, Algebr. Geom. Topol. \textbf{13} (2013),
  no.~3, 1243--1271. \MR{3071127}

\bibitem[Wei99]{WeissImmersion}
M.~Weiss, \emph{Embeddings from the point of view of immersion theory. {I}},
  Geom. Topol. \textbf{3} (1999), 67--101. \MR{1694812}

\bibitem[Wei04]{WeissHomology}
M.~S. Weiss, \emph{Homology of spaces of smooth embeddings}, Q. J. Math.
  \textbf{55} (2004), no.~4, 499--504. \MR{2104688}

\bibitem[Wei05]{WeissClassifying}
M.~Weiss, \emph{What does the classifying space of a category classify?},
  Homology Homotopy Appl. \textbf{7} (2005), no.~1, 185--195. \MR{2175298}

\bibitem[Wei11]{WeissImmersionErratum}
\bysame, \emph{Erratum to the article {E}mbeddings from the point of view of
  immersion theory: {P}art {I}}, Geom. Topol. \textbf{15} (2011), no.~1,
  407--409. \MR{2776849}

\end{thebibliography}

\vspace{+0.2cm}
\end{document}